\documentclass{amsart}
\usepackage{amssymb, graphics, color, enumitem, mathrsfs}
\usepackage[all]{xy}

 \usepackage[applemac]{inputenc}
\usepackage[cyr]{aeguill}

\usepackage[margin = 1in]{geometry}

\newtheorem{lemma}{Lemma}[section]
\newtheorem{proposition}[lemma]{Proposition}
\newtheorem{corollary}[lemma]{Corollary}
\newtheorem{theorem}[lemma]{Theorem}

\newtheorem{example}[lemma]{Example}
\newtheorem{definition}[lemma]{Definition}
\newtheorem{remark}[lemma]{Remark}

\newtheorem*{Acknowledgement}{Acknowledgements}

\newtheorem{Theorem}{Theorem}
\newtheorem*{Remark}{Remark}
\newtheorem{assumption}[lemma]{Assumption}

\newcommand\cf{cf\@. }

\newcommand\pa{ \partial}

\newcommand\bbC{\mathbb C}

\newcommand\bbN{\mathbb N}
\newcommand\bbP{\mathbb P}

\newcommand\bbR{\mathbb R}
\newcommand\bbS{\mathbb S}

\newcommand\bbZ{\mathbb Z}

\renewcommand\Im{\operatorname{Im}}

\usepackage{color}

\newcommand\CI{\mathcal{C}^{\infty}}

\newcommand\fc{\operatorname{fc}}

\newcommand\cA{\mathcal{A}}

\newcommand\cL{\mathcal{L}}

\newcommand\cV{\mathcal{V}}
\newcommand\cU{\mathcal{U}}
\newcommand\cI{\mathcal{I}}

\newcommand\pr{\operatorname{pr}}

\newcommand\Id{\operatorname{Id}}

\newcommand\WH{\operatorname{WH}}

\newcommand\cH{\mathcal{H}}

\newcommand\tM{\widetilde{M}}
\newcommand\hM{\widehat{M}}

\newcommand\cN{\mathcal{N}}
\newcommand\cM{\mathcal{M}}

\newcommand\SU{\operatorname{SU}}
\newcommand\IH{\operatorname{IH}}
\newcommand\Sp{\operatorname{Sp}}

\newcommand{\Hilb}{\operatorname{Hilb}}
\newcommand{\tHilb}{\widetilde{\Hilb}}

\newcommand\QAC{\operatorname{QAC}}
\newcommand\QCyl{\operatorname{QCyl}}
\newcommand\Qb{\operatorname{Qb}}
\newcommand\QFB{\operatorname{QFB}}
\newcommand\QFC{\operatorname{QFC}}
\newcommand\ALE{\operatorname{ALE}}
\newcommand\QALE{\operatorname{QALE}}

\newcommand\SL{\operatorname{SL}}

\newcommand\bd{\operatorname{b}}

\newcommand\cB{\mathcal{B}}

\newcommand\mf{\mathfrak}
\newcommand\hU{\widehat{\cU}}
\newcommand\hV{\widehat{\cV}}
\newcommand\cW{\mathcal{W}}

\newcommand\bbT{\mathbb{T}}
\newcommand\tcM{\widetilde{\cM}}
\newcommand\tcN{\widetilde{\cN}}
\newcommand\bcM{\overline{\cM}}
\newcommand\hcM{\widehat{\cM}}
\newcommand\hcN{\widehat{\cN}}
\newcommand\cP{\mathcal{P}}

\begin{document}
\title[$L^2$-cohomology of $\QFB$-metrics]
{$L^2$-cohomology of quasi-fibered boundary metrics}

\author{Chris Kottke}
\address{New College of Florida}
\email{ckottke@ncf.edu}

\author{Fr\'ed\'eric Rochon}
\address{Département de Mathématiques, Universit\'e du Qu\'ebec \`a Montr\'eal}
\email{rochon.frederic@uqam.ca}

\maketitle

\begin{abstract}
We develop new techniques to compute the weighted $L^2$-cohomology of quasi-fibered boundary metrics (QFB-metrics).  Combined with the decay of $L^2$-harmonic forms obtained in a companion paper, this allows us to compute the reduced $L^2$-cohomology for various classes of QFB-metrics.  Our results applies in particular to the Nakajima metric on the Hilbert scheme of $n$ points on $\bbC^2$, for which we can show that the Vafa-Witten conjecture holds.  Using the compactification of the monopole moduli space announced by Fritzsch, the first author and Singer, we can also give a proof of the Sen conjecture for the monopole moduli space of magnetic charge $3$.     
\end{abstract}

\tableofcontents

\section{Introduction}

In \cite{Joyce}, Joyce constructed complete Calabi-Yau metrics on crepant resolutions of $\bbC^n/\Gamma$ for $\Gamma\subset \SU(n)$ a finite subgroup.  When $\Gamma$ acts freely on $\bbC^n\setminus \{0\}$, this is an Asymptotically Locally Euclidean (ALE) metric.  However, when the action is not free away from the origin, the orbifold $\bbC^n/ \Gamma$ has rays of singularities going off to infinity and the metrics are only ALE away from these singularities, hence the name Quasi-Asymptotically Euclidean (QALE) introduced by Joyce.  To construct these metrics, Joyce solved a complex Monge-Ampère equation for an appropriate choice of  Kähler QALE-metric using some mapping properties of the corresponding Laplacian.  More recently, Mazzeo and Degeratu in \cite{DM2018} have introduced the notion of Quasi-Asymptotically Conical (QAC) metrics, essentially generalizing QALE-metrics in the same way that Asymptotically Conical (AC) metrics generalize ALE-metrics, and showed that the corresponding Laplacian is Fredholm when acting on suitable weighted Hölder or Sobolev spaces.    The definition of \cite{DM2018} was in terms of resolution blowups.  In \cite{CDR}, a coordinate-free definition of $\QAC$-metrics was provided in terms of a Lie structure at infinity and a natural compactification by a manifold with corners.  Using this point of view, the first examples of Calabi-Yau QAC-metrics that are not QALE or AC were constructed in \cite{CDR}.  More generally, the notion of Quasi-Fibered Boundary (QFB) metrics was introduced in \cite{CDR} by adding a compact fiber at infinity.  

Natural examples of hyperKähler QFB-metrics appear on moduli spaces.  For instance, on the Hilbert scheme $\operatorname{Hilb}_0^n(\bbC^2)$, it was shown in \cite{Carron2011} that the Nakajima metric is in fact a QALE-metric.  In \cite{KS}, a partial compactification of the moduli space of $\SU(2)$-monopoles on $\bbR^3$ is obtained via a gluing construction with the corresponding $L^2$-metric behaving like a QFB-metric in that direction.  In fact,  a full compactification of the monopole moduli space was announced in \cite{FKS} with the property that the corresponding $L^2$-metric is a QFB-metric.  Similarly, on the moduli space of $\SL(2,\bbC)$-Higgs bundles, a polynomial convergence at infinity of the $L^2$-metric towards the semi-flat metric was obtained in the regular part of the Hitchin system in \cite{MSWW2017}, a result that was subsequently improved to an exponential convergence in \cite{Dumas-Neitzke,Frederickson}.  In this latter setting, it is expected that the $L^2$-metric should be like a QFB-metric, but with some singular fibers at infinity.  

All these efforts to understand the asymptotic of these hyperKähler metrics were in part driven by various open conjectures about their Hodge cohomology, that is, their space of $L^2$-harmonic forms.  The first conjecture of the sort is the Sen conjecture \cite{Sen} coming from string theory and S-duality, which predicts that the Hodge cohomology  of the $L^2$-metric of the universal cover $\widetilde{\cM}^0_k$ of the reduced moduli space $\cM_k^0$ of  $\SU(2)$-monopoles of magnetic charge $k$ on $\bbR^3$ is only non-trivial in middle degree and admits a complete description in terms of a natural $\bbZ_k$-action.  More precisely, if $\mathcal{H}^{q}_{p}(\widetilde{\cM}_k^0)$ denotes the space of $L^2$-harmonic forms of degree $q$ and weight $p$ with respect to the $\bbZ_k$-action, then the Sen conjecture predicts in middle degree $q=2k-2$ that $\mathcal{H}^{2k-2}_{p}(\widetilde{\cM}_k^0)\cong \bbC$ if $k$ and $p$ are coprime and $\mathcal{H}^{2k-2}_{p}(\widetilde{\cM}_k^0)=0$ otherwise.  Soon after the formulation of the conjecture,  Segal and Selby in \cite{Segal-Selby} computed the relative and absolute cohomologies $H^*_c(\widetilde{\cM}_k^0)$ and $H^*(\widetilde{\cM}_k^0)$ of the universal cover of the reduced moduli space and gave supporting evidence to the conjecture. 
Indeed, they showed that the images $\Im\bigl[H^q_c(\widetilde{\cM}_k^0) \rightarrow H^q(\widetilde{\cM}_k^0)\bigr]$ satisfy the predictions of Sen's conjecture, and since
these images factor through the space of $L^2$-harmonic forms as composites of natural maps $H^q_c(\widetilde{\cM}_k^0) \rightarrow \cH^q(\widetilde{\cM}_k^0) \rightarrow H^q(\widetilde{\cM}_k^0)$,
it follows that $\cH^*(\widetilde{\cM}_k^0)$ is no smaller than Sen's prediction \cite[Sentence after Theorem~1.3]{Segal-Selby}. In fact, since Segal and Selby prove that $H^q_c(\widetilde{\cM}_k^0) \to H^q(\widetilde{\cM}_k^0)$ is an isomorphism in middle degree and trivial otherwise, the composite maps split and Sen's conjecture can be reformulated as saying that the associated inclusion
\begin{equation}
       \Im\left[ H^q_c(\cM_k^0)\to H^q(\cM_k^0)\right]  \hookrightarrow \cH^q(\cM_k^0)
\label{eq.1} \end{equation}
is in fact an isomorphism.  A major step toward a proof of the conjecture was subsequently made by Hitchin \cite{Hitchin}, who showed that for many hyperKähler metrics coming from hyperKähler quotients, in particular the $L^2$-metric on $\widetilde{\cM}^0_k$, the $L^2$-harmonic forms all lie in middle degree, giving a proof of the conjecture outside the middle degree and a full proof of the conjecture when $k=2$.  In this specific case, an alternative proof of the Sen conjecture was also subsequently obtained by Hausel-Hunsicker-Mazzeo \cite{HHM2004}. Since then, the unsettled part of the conjecture, that is, whether or not in middle degree the map \eqref{eq.1} is surjective, has remained open for $k\ge 3$.      

Shortly after Sen formulated his conjecture,  Vafa and Witten in \cite[Discussion after equation (4.43)]{Vafa-Witten} (see also \cite[Conjecture~1.4]{Carron2011}) made a similar S-duality prediction for the Hodge cohomology of quiver varieties.  For the Nakajima metric on $\Hilb_0^n(\bbC^2)$, their conjecture predicts that the natural inclusion
\begin{equation}
 \Im \left[ H^q_c(\Hilb_0^n(\bbC^2))\to H^q(\Hilb^n_0(\bbC^2)) \right]  \hookrightarrow   \cH^q(\Hilb^n_0(\bbC^2))
\label{eq.2}\end{equation}
constructed as above,
must be surjective.  Again, by the result of \cite{Hitchin}, it is automatically true except possibly in middle degree, and the argument in \cite[\S 5.5]{Hitchin} gives a complete proof of the conjecture when $n=2$.  Alternatively, when $n=2$, a proof of the conjecture follows from standard results about the $L^2$-cohomology of AC-metrics, see for instance \cite{MelroseGST} or \cite[Theorem~1A]{HHM2004}.  When $n=3$ instead, the conjecture follows from the computation by Carron \cite{Carron2011b}  of the Hodge cohomology of QALE-metrics of depth 2.  For the moduli space of $\SL(2,\bbC)$-Higgs bundles,
Hausel showed in \cite{Hausel} that the image of relative cohomology into absolute cohomology is trivial, so that inspired by \cite{Segal-Selby}, he was led to conjecture that the Hodge cohomology should be trivial.  Again in this case, the results of \cite{Hitchin} prove this conjecture except in middle degree.

In the present paper, we derive new results about the Hodge cohomology of $\QFB$-metrics.  In particular, we obtain the following advances on the Sen conjecture and the Vafa-Witten conjecture.

\begin{Theorem}
The Sen conjecture holds on $\widetilde{\cM}^0_3$ provided the natural $L^2$-metric on $\widetilde{\cM}^0_3$ is a $\QFB$-metric as announced in \cite{FKS}.  
\label{int.1}\end{Theorem}

\begin{Theorem}
The Vafa-Witten conjecture holds on $\Hilb^n_0(\bbC^2)$ for all $n\ge 2$.
\label{int.2}\end{Theorem}
\begin{Remark}
Relying on a different approach, a proof of the Vafa-Witten conjecture on $\Hilb^n_0(\bbC^2)$ for all $n$ was announced by Melrose in \cite{Melrose_CIRM}.
\label{int.3}\end{Remark}

Our general strategy to prove such a result is strongly inspired by the work of Hausel-Hunsicker-Mazzeo \cite{HHM2004}, where a complete description of the Hodge cohomology of fibered boundary and fibered cusp metrics in terms of intersection cohomology was obtained.  As the name suggests, fibered boundary metrics are a particular case of $\QFB$-metrics.  Let us recall that on a compact manifold $M$ with boundary $\pa M$ equipped with a fiber bundle $\phi: \pa M\to Y$ over a closed manifold $Y$ and a tubular neighborhood $c: \pa M\times [0,\delta)\to M$ of $\pa M$, an example of fibered boundary metric is given by a Riemannian metric $g_{\phi}$ on $M\setminus \pa M$ such that 
$$
    c^*g_{\phi}= \frac{dx^2}{x^4}+ \frac{\phi^*g_Y}{x^2}+ \kappa,
$$ 
where $x$ is the coordinate on the factor $[0,\delta)$, $g_{Y}$ is a Riemannian metric on $Y$ and $\kappa$ is a symmetric $2$-tensor on $\pa M$ such that $\phi^*g_Y+ \kappa$ is a Riemannian metric on $\pa M$ making $\phi: \pa M\to Y$ a Riemannian submersion with respect to the Riemannian metrics $\phi^*g_Y+\kappa$ and $g_Y$.  On the other hand, a fibered cusp metric $g_{\fc}$ is a complete metric on $M\setminus \pa M$ which near $\pa M$ is of the form 
$$
      g_{\fc}= x^2 g_{\phi}
$$
for some fibered boundary metric $g_{\phi}$.  For instance, for $g_Y$ and $\kappa$ as above, an example of fibered cusp metric $g_{\fc}$ is given by one such that 
$$
     c^*g_{\fc}= \frac{dx^2}{x^2}+ \phi^*g_Y+ x^2\kappa.
$$  

The main result of \cite{HHM2004} roughly relies on two intermediate results.  The first one, topological in nature, consists in identifying weighted $L^2$-cohomology groups of fibered boundary and  fibered cusp metrics with suitable intersection cohomology groups.  Since fibered cusp and fibered boundary metrics are conformally related, one in fact only needs to establish such a result for the weighted $L^2$-cohomology of a fibered cusp metric.  These weighted $L^2$-cohomology groups can be understood as the cohomology groups of a sheaf on an associated stratified space, so that using Mayer-Vietoris long exact sequences, it suffices to identify these weighted $L^2$-cohomology groups with intersection cohomology for local models.  These local identifications in turn can be achieved thanks to the K\"unneth formula of Zucker \cite[Corollary~2.34]{Zucker} for the $L^2$-cohomology of warped products.    Except for certain types of fibered cusp metrics, notice that $L^2$-cohomology itself is infinite dimensional and cannot be identified with some intersection cohomology, hence the importance to introduce a weight to obtain such an identification in general.
The second intermediate result, more analytical in nature, consists in showing that a $L^2$-harmonic form with respect to a fibered boundary or a fibered cusp metric admits a polyhomogeneous expansion at infinity, so that in particular it decays a bit faster compared to a general $L^2$-form.  This can be established thanks to the pseudodifferential calculus of Mazzeo-Melrose \cite{Mazzeo-MelrosePhi} and the parametrix construction of Vaillant \cite{Vaillant}.  These also allow to show that the Hodge-deRham operator is Fredholm when acting on suitable Sobolev spaces, a result which is also used in \cite{HHM2004}.

It turns out that both of these intermediate results can be suitably adapted to study the Hodge cohomology of  $\QFB$-metrics.  First, to compute the weighted $L^2$-cohomology of a $\QFB$-metric, we can in fact introduce the analog of fibered cusp metrics, namely the notion quasi-fibered cusp metrics ($\QFC$-metrics), a class of metrics conformally related to $\QFB$-metrics, see Definition~\ref{wl2.1} below for more details.  Quasi-fibered cusp metrics should not be confused with the class of iterated fibered cusp metrics considered in \cite{DLR,HR}, which constitutes yet another way of generalizing the notion of fibered cusp metrics to stratified spaces of higher depth.  In fact, one important difference is that as opposed to iterated fibered cusp metrics, $\QFC$-metrics do not admit nice local models in terms of warped products.  In particular, the K\"unneth formula of Zucker for the $L^2$-cohomology of a warped product cannot be used to compute the $L^2$-cohomology of local models of $\QFC$-metrics.  As described in \eqref{qfb.11} below, a good local model for $\QFB$-metrics is not given by a warped product, but by a \textbf{subset} of a Cartesian product.  To compute the weighted $L^2$-cohomology of such a local model, our main tool consists instead in using basic mapping properties of the exterior differential on a half-line.  As in \cite{HHM2004}, we need to show that we can define weighted $L^2$-cohomology using a sheaf of conormal forms, which can be achieved using a soft parametrix inverting the Hodge Laplacian of a $\QFB$-metric, for instance using the general pseudodifferential calculus of \cite{ALN2007} for Lie structures at infinity.  With this approach, we can in many cases identified weighted $L^2$-cohomology groups with some intersection cohomology groups.  In the following case, we can even compute the Hodge cohomology of a $\QFC$-metric (see also Corollary~\ref{wl2.41} below for more details).
\begin{Theorem}
Let $g_{\QFC}$ be a $\QFC$-metric on the regular part of a smoothly stratified space $\hM$.  Suppose that  the link $\widehat{Z}$ of any singular stratum $\widehat{S}$ of $\hM$ is odd dimensional and is such that  
$$
       \IH_{\underline{\mathfrak{m}}}^{\frac{\dim \widehat{Z}+1}2}(\widehat{Z})=\{0\},
$$
where $ \IH_{\underline{\mathfrak{m}}}^{q}(\widehat{Z})$ denotes the intersection cohomology group of lower middle perversity in degree $q$.  Then the $L^2$-cohomology of $g_{\QFC}$ and its Hodge cohomology are both naturally identified with $\IH^*_{\underline{\mathfrak{m}}}(\hM)$.  
\label{int.4}\end{Theorem}  
In \cite{HR}, a similar identification was obtained in the case of an iterated fibered cusp, though in this latter case one only needs the weaker and simpler assumption that the smoothly stratified space $\hM$ be Witt.

For the computation of the Hodge cohomology of $\QFB$-metrics, more analysis must be involved, especially since already for fibered boundary metrics, the $L^2$cohomology is infinite dimensional and distinct from the Hodge cohomology.  In particular, to study the asymptotic behavior of $L^2$-harmonic forms of a $\QFB$-metric, we have developed in the companion paper \cite{KR1} a pseudodifferential calculus suitable to construct parametrices for  the Hodge-deRham operator of a $\QFB$-operator.  Such parametrices allow us, for appropriate $\QFB$-metrics, to show that $L^2$-harmonic forms decay at infinity a bit faster than a general $L^2$-form.  The specific result that we will invoke from \cite{KR1} is stated in Theorem~\ref{qfb.12} below.  

This decay of $L^2$-harmonic forms and our computation of the weighted $L^2$-cohomology of a $\QFC$-metric
allow us to compute the Hodge cohomology of various $\QFB$-metrics.  In order to do this, we diverge from the approach used in the proof of \cite[Theorem~1C]{HHM2004}.  Instead, we propose a softer argument by considering the inverse of the map considered in \cite[Theorem~1C]{HHM2004}, e.g., in the case of the Sen conjecture and the Vafa-Witten conjecture, we consider the maps \eqref{eq.1} and \eqref{eq.2}.  Indeed, in \cite{HHM2004}, the Fredholmness of the Hodge-deRham operator is used by the authors to show the surjectivity of their map.  The analogue in our setting corresponds to establishing the injectivity of our map, which can be achieved following the approach of Segal and Selby \cite[Lemma~1.4]{Segal-Selby}, see for instance the proof of Theorem~\ref{hs.12} or Proposition~\ref{mms.5} below.  On the other hand, the surjectivity of our map follows almost immediately from the decay of $L^2$-harmonic forms.  

In the particular setting of Theorem~\ref{int.1} and Theorem~\ref{int.2}, we also need to identify our weighted $L^2$-cohomology groups, already identified with some intersection cohomology groups, with the usual relative and absolute cohomology groups.  For $\Hilb^n_0(\bbC^2)$, this follows essentially from a result of Nakajima \cite[Corollary~5.10]{Nakajima} asserting that $\Hilb^n_0(\bbC^2)$ has no absolute cohomology above the middle degree, see Corollary~\ref{sa.10} below.     For $\widetilde{M}^0_3$, we have a similar identification, but only in middle degree and using a more intricate argument relying instead on \cite{Segal-Selby}; see Proposition~\ref{mms.14} below.  Unfortunately, this argument does not seem to extend to $\widetilde{M}^0_k$ for $k\ge 4$; see the discussion after Theorem~\ref{mms.30} for more details. Our approach can also be used to give a new proof and a generalization of the results of Carron in \cite{Carron2011b}, see in particular Theorem~\ref{do.6} and Corollary~\ref{do.9}.

The paper is organized as follows.  In \S~\ref{qfb.0}, we review basic facts about $\QFB$-metrics and we state the result from \cite{KR1} that we will used.  In \S~\ref{wl2.0}, we introduce $\QFC$-metrics and derive our results about their weighted $L^2$-cohomology groups.  This is applied to the Nakajima metric on $\Hilb^n_0(\bbC^2)$ in \S~\ref{hs.0},  where we derive in particular a proof of Theorem~\ref{int.2}. In \S~\ref{do.0}, this is applied instead to a large class of $\QAC$-metrics of depth 2 including those considered by Carron in \cite{Carron2011b}.  Finally, in \S~\ref{mms.0}, our results about weighted $L^2$-cohomology are used to study the Hodge cohomology on the reduced $\SU(2)$-monopole moduli space of charge $3$ and prove Theorem~\ref{int.1}.  
\begin{Acknowledgement}
The authors are grateful to Rafe Mazzeo, Richard Melrose and Michael Singer for helpful discussions, as well as to an anonymous referee for useful suggestions. CK was supported by NSF Grant number DMS-1811995.  This paper is also based in part on work supported by NSF under grant DMS-1440140 while CK was in residence at the Mathematical Sciences Research Institute in Berkeley, California during the Fall 2019.  FR was supported by NSERC and a Canada research chair.
\end{Acknowledgement}

\numberwithin{equation}{section}

\section{Quasi-fibered boundary metrics} \label{qfb.0}

Let $M$ be a compact manifold with corners in the sense of \cite{Melrose1992, MelroseMWC,Grieser}.  Let $H_1,\ldots, H_{\ell}$ be a complete list of its boundary hypersurfaces.  

\begin{definition}[\cite{AM2011,ALMP2012,DLR}]
  Let $\phi=\{\phi_1,\ldots,\phi_{\ell}\}$ be a collection of fiber bundles $\phi_i:H_i\to S_i$ over a compact manifold with corners $S_i$.   We say that $\phi$ is an \textbf{iterated fibration structure} for $M$ if there is a partial order on the boundary hypersurfaces of $M$ such that 
\begin{itemize}
\item Any subset $\cI$ of boundary hypersurfaces such that $\cap_{i\in\cI} H_i\ne \emptyset$ is totally ordered;
\item If $H_{i}<H_j$, then $H_i\cap H_j\ne \emptyset$, $\phi_i|_{H_{i}\cap H_j}: H_i\cap H_j\to S_i$ is a surjective submersion, $S_{ji}:= \phi_j(H_j\cap H_i)$ is a boundary hypersurface of the manifold with corners $S_j$ and there is a surjective  submersion $\phi_{ji}: S_{ji}\to S_i$ such that $\phi_{ji}\circ \phi_j=\phi_i$ on $H_i\cap H_j$;
\item The boundary hypersurfaces of $S_j$ are given by $S_{ji}$ for $H_i<H_j$.
\end{itemize}
In this case, we say that the pair $(M,\phi)$ is a \textbf{manifold with fibered corners}.
\label{qfb.1}\end{definition}

If $(M,\phi)$ is a manifold with fibered corners, then for each fiber bundle $\phi_j: H_j\to S_j$, the base $S_j$ automatically inherits an iterated fibration structure specified by the fiber bundles $\phi_{ji}: S_{ji}\to S_i$ for each $i$ with $H_i<H_j$.  Each fiber of $\phi_j:H_j\to S_j$ has also a natural iterated fibration structure induced by the fiber bundles $\phi_i: H_i\to S_i$ for each $H_i>H_j$.   

As described in \cite{ALMP2012, DLR}, collapsing the fibers of $\phi_i$ onto their base for each boundary hypersurface $H_i$ yields a smoothly stratified space $\hM_{\phi}$.  In particular, for each boundary hypersurface $H_i$, one can associate to $S_i$ a smoothly stratified space $\widehat{S}_i$.  Each $\widehat{S}_i$ is naturally included in $\hM_{\phi}$ as a singular stratum and all the singular strata of $\hM_{\phi}$ arise in this way.  The depth of the stratified space $\hM_{\phi}$ then corresponds to the depth of the manifold with corners $M$, which can be defined as  the highest possible codimension of a corner of $M$.      

On a manifold with fibered corners $(M,\phi)$, a \textbf{maximal} boundary hypersurface is a boundary hypersurface which is maximal with respect to the partial order.  A boundary hypersurface which is not maximal is said to be \textbf{non-maximal}.  A \textbf{submaximal} boundary hypersurface of $H_j$ is a non-maximal boundary hypersurface such that 
$$
     H_i>H_j\quad \Longrightarrow \quad H_i \; \mbox{is maximal}.
$$
For each boundary hypersurface $H_i$, let $x_i\in \CI(M)$ be a boundary defining function, that is, $x_i$ takes nonnegative values, $x_i^{-1}(0)= H_i$ and $dx_i$ is nowhere zero on $H_i$.  We say that $x_i$ is \textbf{compatible} with the iterated fibration structure $\phi$ if $x_i$ restricted to $H_j$ is constant in the fibers of $\phi_j: H_j\to S_j$ whenever $H_j>H_i$.  In this paper, we will always assume that our boundary defining functions are compatible with the iterated fibration structure $\phi$.  This obviously imposes restrictions on the type of boundary defining functions we will consider, but no restriction on the type of manifolds with fibered corners by \cite[Lemma~1.4]{DLR}.

\begin{definition}
Let $v=\prod_i x_i$ be a total boundary defining function for the manifold with fibered corners $(M,\phi)$.  The space $\cV_{\QFB}(M)$ of \textbf{quasi-fibered boundary vector fields} ($\QFB$-vector fields) consists in smooth vector fields $\xi$ in $M$ such that 
\begin{itemize}
\item[(i)] $\xi$ is tangent to the fibers of $\phi_i: H_i\to S_i$ for each boundary hypersurface $H_i$ of $M$;
\item[(ii)] $\xi v\in v^2\CI(M)$.  
\end{itemize}
\label{qfb.2}\end{definition}
\begin{remark}
As explained in \cite{KR1}, this definition is equivalent to the more complicated definition originally provided in \cite{CDR}.
\label{qfb.2new}\end{remark}
Notice that a $\QFB$-vector field is in particular tangent to $H_i$ for each boundary hypersurface, so $\cV_{\QFB}(M)$ is a subspace of the Lie algebra of $b$-vector fields of Melrose \cite{MelroseAPS},
$$
      \cV_b(M)= \{\xi\in \CI(M;TM)\; | \; \xi x_i\in x_i\CI(M) \; \forall i\}.
$$
In fact, just imposing condition (i) gives the Lie algebra of edge vector fields $\cV_e(M)$ of Mazzeo \cite{MazzeoEdge, ALMP2012,AG}.

Condition (ii) on the other hand clearly depends on the choice of $v$.  By \cite[Lemma~1.1]{KR1}, two total boundary defining functions $v$ and $v'$ will give the same space of $\QFB$-vector fields if and only if the function $\frac{v}{v'}$ is constant on the fibers of $\phi_i$ for each boundary hypersurface $H_i$.  Clearly, conditions (i) and (ii) are preserved by the Lie bracket, so that $\cV_{\QFB}(M)$ is a Lie subalgebra of $\cV_b(M)$.  As explained in \cite{CDR}, there is in fact a natural bundle ${}^{\QFB}TM\to M$, called the $\QFB$-tangent bundle, and a natural map 
\begin{equation}
    a: {}^{\QFB}TM\to TM
\label{qfb.3}\end{equation}
inducing a canonical inclusion 
\begin{equation}
   a_*: \CI(M;{}^{\QFB}TM)\to \cV_{\QFB}(M).
\label{qfb.4}\end{equation}
This gives ${}^{\QFB}TM$ the structure of a Lie algebroid with anchor map \eqref{qfb.3}.  The anchor map is not an isomorphism of vector bundles, but it becomes one when restricted to the interior of $M$,
\begin{equation}
      a:  {}^{\QFB}TM|_{M\setminus \pa M} \tilde{\longrightarrow} T(M\setminus \pa M).
\label{qfb.5}\end{equation}
\begin{definition}[\cite{CDR,KR1}]
A \textbf{quasi-fibered boundary metric} ($\QFB$-metric) for a manifold with fibered corners $(M,\phi)$ and a choice of total boundary defining function $v$ is a Riemannian metric on the interior of $M$ which is of the form 
$$
    a_*(h|_{M\setminus \pa M})
$$
for some choice of bundle metric $h\in \CI(M;S^2({}^{\QFB}T^*M))$ for the vector bundle ${}^{\QFB}TM$.  In this case, we say that the manifold with corners $M$ is the \textbf{$\QFB$-compactification} of the corresponding Riemannian manifold.   When $(M,\phi)$ is such that for each maximal boundary hypersurface $H_i$, $S_i=H_i$ and $\phi_i$ is the identity map, a $\QFB$-metric is also said to be a \textbf{quasi-asymptotically conical} metric ($\QAC$-metric), in which case $M$ is also said to be the $\QAC$-compactification of the corresponding Riemannian manifold.    
\label{qfb.6}\end{definition}
A good measure of the complexity of a $\QFB$-metric is its \textbf{depth}, which we take\footnote{A slightly different convention is used in \cite{DM2018}.} to be the depth of the underlying manifold with fibered corners.
When $g_{\QFB}$ is a $\QFB$-metric, the pair $(M\setminus\pa M,g_{\QFB})$ is a particular example of Riemannian manifold with Lie structure at infinity in the sense of \cite{ALN04}.  As such, a $\QFB$-metric is a complete Riemannian metric of infinite volume with curvature and all its covariant derivatives bounded.  By \cite[Proposition~1.27]{CDR} or \cite{Bui}, the injectivity radius of $\QFB$-metric is bounded below by a positive constant, so that $\QFB$-metrics have bounded geometry.  

Similarly, to the Lie algebra of edge vector fields, one can associate the edge tangent bundle ${}^eTM\to M$ and the class of edge metrics on $M\setminus \pa M$.  Again, for $g_e$ an edge metric, the pair $(M\setminus \pa M,g_e)$ is a Riemannian manifold with Lie structure at infinity.  To $(M,\phi)$, one can yet associated a third class of metrics, namely the class of \textbf{wedge} metrics (also called incomplete iterated edge metrics in \cite{ALMP2012}), given by metrics $g_w$ on $M\setminus\pa M$  of the form
$$
 g_w= v^2g_e
$$  
for some edge metric $g_e$.  A wedge metric $g_w$ is of finite volume and is geodesically incomplete, so the pair $(M\setminus \pa M,g_w)$ is not a Riemannian manifold with Lie structure at infinity.  

In this paper, our main interest in wedge metrics is that they can be used to construct simple examples of $\QFB$-metrics as described in \cite[\S~1]{KR1}.  Indeed, let 
\begin{equation}
     c_i: H_i\times [0,\delta_i)\to M
\label{qfb.7}\end{equation}
be a collar neighborhood of $H_i$ compatible with the boundary defining functions in the sense that $c_i^*x_i$ corresponds to the projection $H_i\times [0,\delta_i)\to [0,\delta_i)$ and $c^*_ix_j$ is the pull-back of a function on $H_i$ for $j\ne i$.  Instead of the level sets of $x_i$, we can use the level sets of the total boundary defining function $v$, that is, consider the open set
\begin{equation}
   \cU_i=\{ (p,\tau)\in H_i\times [0,\delta_i) \; | \; \Pi_{j\ne i}  x_j(p)> \frac{\tau}{\delta_i}\subset H_i\times [0,\delta_i)
\label{qfb.8}\end{equation}
with natural diffeomorphism 
\begin{equation}
\begin{array}{llcl}
   \psi_i: & H_i\setminus \pa H_i\times [0,\delta_i) & \to & \cU_i \\
               & (p,t) & \mapsto & (p, t\prod_{j\ne i} x_j(p)).
\end{array}
\label{qfb.9}\end{equation}

On $\cU_i$ seen as a subset of $H_i\times [0,\delta_i)$, let $\pr_1$ and $\pr_2$ be the restrictions of the projections of $H_i\times [0,\delta_i)$ onto $H_i$ and $[0,\delta_i)$.  Choose a connection for the fiber bundle $\phi_i: H_i\to S_i$.  On $S_i\setminus \pa S_i$, let $g_{S_i}$ be a wedge metric compatible with the iterated fibration structure of $S_i$.  Let $\kappa_i$ be a family of fiberwise $\QFB$-metrics in the fibers of $\phi_i:H_i\to S_i$.  Using the connection of $\phi_i$, this can be lifted to a vertical symmetric $2$-tensor on $H_i\setminus \pa H_i$.  In $\cU_i$, still seen as a subset of $H_i\times [0,\delta_i)$, an example of $\QFB$-metric is then given by 
\begin{equation}
      g_{\QFB}= \frac{d\tau^2}{\tau^4}+ \frac{\pr_1^*\phi_i^*g_{S_i}}{\tau^2}+ \pr_1^*\kappa_i.
\label{qfb.10}\end{equation} 
When the fiber bundle $\phi_i: H_i\to S_i$ is trivial with $H_i=S_i\times Z_i$ and $\kappa_i=g_{Z_i}$  is a constant family of $\QFB$-metrics in the fiber $Z_i$, the example \eqref{qfb.10} corresponds to a Cartesian product of the $\QFB$-metric $g_{Z_i}$ and the Riemannian cone 
$$
       \frac{d\tau^2}{\tau^4}+ \frac{g_{S_i}}{\tau^2}
$$  
with cross-section $(S_i\setminus\pa S_i, g_{S_i})$.  One important subtlety is that $\cU_i$ is only a \textbf{subset} of this Cartesian product.  In particular, for fixed $\tau$ and $s\in S_{i}\setminus \pa S_i$, we only consider the metric $g_{Z_i}$ in the region of $Z_i$ where $\prod_{j\ne i}x_j>\frac{\tau}{\delta_i}$.  Still, a $\QFB$-metric of the form \eqref{qfb.10} near $H_i$ is said to be of \textbf{product-type} near $H_i$.  More generally, an exact $\QFB$-metric is a $\QFB$-metric which is product-type near $H_i$ up to a term in $x_i\CI(M;S^2({}^{\QFB}T^*M))$ for each boundary hypersurface $H_i$ of $M$.  

Coming back to the model \eqref{qfb.10}, on an open set $B_i\subset S_i\setminus \pa S_i$ over which the fiber bundle $\phi_i$ is trivial, an example of $\QFB$-metric on $\cU_i\cap (\phi_i^{-1}(B_i)\times [0,\delta_i))$ is given by the restriction of the Cartesian product 
\begin{equation}
   \frac{d\tau^2}{\tau^4}+ \frac{g_{B_i}}{\tau^2}+ g_{Z_i}
\label{qfb.11}\end{equation} 
with $g_{B_i}$ a Riemannian metric on $B_i$ and $g_{Z_i}$ a $\QFB$-metric on $Z_i$.  By the Ehresmann lemma of \cite[Corollary~A.6]{KR3}, a $\QFB$-metric is always locally quasi-isometric to a model of the form \eqref{qfb.11}.  Since $Z_i$ is of lower depth, one can use \eqref{qfb.11} to define $\QFB$-metrics iteratively.  This is how the subclass of $\QAC$-metrics was originally introduced in \cite{DM2018}.  

For the computation of weighted $L^2$-cohomology of a $\QFB$-metric in the next section, the local model \eqref{qfb.11} is essentially all that we will use about $\QFB$-metrics.  However, to obtain results about the reduced $L^2$-cohomology of $\QFB$-metrics, we will need to invoke a result about the decay of harmonic forms obtained in our companion paper \cite{KR1}.  For the convenience of the reader, let us describe in details the specific result of \cite{KR1} that we will use.  

Suppose now that $g_{\QFB}$ is an exact $\QFB$-metric and let $\eth_{\QFB}$ be the corresponding Hodge-deRham operator.  If $H_i$ is maximal, the fibered of $\phi_i: H_i\to S_i$ are closed manifolds, so by Hodge theory and \cite[Proposition~15]{HHM2004}, the space of fiberwise harmonic forms on $\phi_i:H_i\to S_i$ with respect to the fiberwise metric $\kappa_i$ form a flat vector bundle $\cH^*_{L^2}(H_i/S_i)$ over $S_i$.  Let $\mathfrak{d}_{S_i}$ be the Hodge-deRham operator associated to the metric $g_{S_i}$ and acting on $\cH^*_{L^2}(H_i/S_i)$-valued forms on $S_i\setminus \pa S_i$.  If $H_i$ is not maximal, the fibers of $\phi_i: H_i\to S_i$ are instead  manifolds with fibered corners and $\kappa_i$ is a family of $\QFB$-metrics.  Suppose that the fiberwise reduced $L^2$-cohomology is finite dimensional.  Thanks to \cite[Theorem~1]{HHM2004}, the corresponding space of fiberwise $L^2$-harmonic forms on $\phi_i: H_i\to S_i$ still form a flat vector bundle $\cH^*_{L^2}(H_i/S_i)$ over $S_i$.  Thus, we can still consider the Hodge-deRham operator $\mathfrak{d}_{S_i}$ associated now to a wedge metric $g_{S_i}$ and acting on $\cH^*_{L^2}(H_i/S_i)$-valued forms on $S_i\setminus \pa S_i$.  For $H_j<H_i$ and for each fiber $Z_{ij}$ of $\phi_{ij}: S_{ij}\to S_j$ let $\mathfrak{d}_{Z_{ij}}$ be the corresponding Hodge-deRham operator acting on forms taking values in $\cH^*_{L^2}(H_i/S_i)$ and associated to the wedge metric $g_{Z_{ij}}$ induced by $g_{S_i}$.  Let also $\mathfrak{P}_j$ be the subset of degrees where the fibers of $\phi_j: H_j\to S_j$ have non-trivial reduced $L^2$-cohomology and let
$$
  \cP_{ij}: \CI(Z_{ij}; \Lambda^*({}^{w}T^*(Z_{ij}))\otimes \cH^*_{L^2}(H_i/S_i))\to \CI(Z_{ij}; \Lambda^*({}^{w}T^*(Z_{ij}))\otimes \cH^*_{L^2}(H_i/S_i))
$$  
be the projection on sections of total degree $q$ (that is, the sum of the degree in the $\Lambda^*({}^{w}T^*(Z_{ij}))$ factor  and the degree in the $ \cH^*_{L^2}(H_i/S_i)$ is equal to $q$ ) such that $q$ or $q+1$ are in $\mathfrak{P}_j$.

\begin{theorem}[Theorem~17.5 in \cite{KR1}]
Let $g_{\QFB}$ be a $\QFB$-metric on $(M,\phi)$ with respect to a total boundary defining function $v$ which is product-type near $H_i$ up to a term in $x_i^2\CI(M;S^2({}^{\QFB}T^*M))$ for each boundary hypersurface $H_i$ of $M$.  For each boundary hypersurface $H_i$, suppose that the fiberwise reduced $L^2$-cohomology in the fibers of $\phi_i: H_i\to S_i$ is finite dimensional, so that the corresponding fiberwise $L^2$-harmonic forms yield a flat vector bundle $\cH^*_{L^2}(H_i/S_i)\to S_i$.  For each boundary hypersurface $H_i$ and for each $H_j\le H_i$, suppose furthermore that 
\begin{equation}
               \left|q-\frac{\dim Z_{ij}}2\right|\le 1 \quad \Longrightarrow \quad \mathfrak{d}_{Z_{ij}} \;\mbox{has trivial $L^2$-kernel in degree $q$}, 
\label{qfb.12a}\end{equation}
where we use the convention that $Z_{ii}:=S_i$.  If $H_i$ is maximal and $H_j$ is submaximal with $H_j<H_i$, then \eqref{qfb.12a} can in fact be replaced by the weaker condition
\begin{equation}
               \left|q-\frac{\dim Z_{ij}}2\right|< 1 \quad \Longrightarrow \quad \mathfrak{d}_{Z_{ij}} \;\mbox{has trivial $L^2$-kernel in degree $q$}.
\label{qfb.12aa}\end{equation}

For $H_j<H_i$, suppose also that
\begin{equation}
\left|q-\frac{\dim Z_{ij}}2\right|\le \frac32 \quad \Longrightarrow \quad  \cP_{ij}(\ker_{L^2_w}\mathfrak{d}_{Z_{ij}})_q \;\mbox{is trivial},
\label{qfb.12b}\end{equation}  
where $(\ker_{L^2_w}\mathfrak{d}_{Z_{ij}})_q$ is the L$^2$-kernel of $\mathfrak{d}_{Z_{ij}}$ for forms of degree $q$ in the $\Lambda^*({}^{w}T^*(Z_{ij}))$ factor.  Then there exists a $\QFB$-metric $\widetilde{g}_{\QFB}$ with respect to the same total boundary defining function $v$ such that its space of $L^2$-harmonic forms is finite dimensional and contained in 
$$
     v^{\epsilon} L^2\Omega^*(M\setminus \pa M,\widetilde{g}_{\QFB})
$$
for some $\epsilon>0$.  
\label{qfb.12}\end{theorem}
\begin{remark}
Since reduced $L^2$-cohomology only depends on the quasi-isometry class of the metric, see for instance \eqref{wl2.8new} below, when Theorem~\ref{qfb.12} applies, the space of $L^{2}$-harmonic forms is finite dimensional, and after possibly changing the $\QFB$-metric $g_{\QFB}$, we can always assume that it is contained in 
$$
     v^{\epsilon} L^2\Omega^*(M\setminus \pa M,g_{\QFB})
$$
for some $\epsilon>0$.  
\label{afb.12new}\end{remark}

As explained in \cite{KR1}, this yields the following two corollaries.
\begin{corollary}[Corollary~17.7 in \cite{KR1}]
Let $g_{\QFB}$ be an exact $\QFB$-metric on $(M,\phi)$ with respect to a total boundary defining function $v$.  Whenever the fibers of $\phi_i: H_i\to S_i$ have non-trivial reduced $L^2$-cohomology with respect to the induced $\QFB$-metrics, suppose that the stratified space $\widehat{S}_i$ corresponding to $S_i$ is a quotient of a sphere by a finite group of isometries, that $\dim S_i\ge 3$ and that $\dim Z_{ij}=\dim S_i-\dim S_j-1>3$ for each boundary hypersurface $H_j<H_i$.  Then there  exists a $\QFB$-metric $\widetilde{g}_{\QFB}$ with respect to the same total boundary defining function $v$ such that its space of $L^2$-harmonic forms is finite dimensional and contained in 
$$
     v^{\epsilon} L^2\Omega^*(M\setminus \pa M,\widetilde{g}_{\QFB})
$$
for some $\epsilon>0$.  
\label{qfb.12c}\end{corollary}

\begin{corollary}[Corollary~17.8 in \cite{KR1}]
Let $g_{\QFB}$ be an exact $\QFB$-metric on $(M,\phi)$ with respect to a total boundary defining function $v$.  Suppose that, except possibly in middle degree, the fibers of $\phi_i: H_i\to S_i$ have trivial reduced $L^2$-cohomology with respect to the induced $\QFB$-metrics.  When it is non-trivial in middle degree, suppose that the 
stratified space $\widehat{S}_i$ corresponding to $S_i$ is a quotient of a sphere by a finite group of isometries, that $\dim S_i\ge 3$ and that $\dim Z_{ij}=\dim S_i-\dim S_j-1\ge 3$ for each boundary hypersurface $H_j<H_i$.
Then there  exists a $\QFB$-metric $\widetilde{g}_{\QFB}$ with respect to the same total boundary defining function $v$ such that its space of $L^2$-harmonic forms is finite dimensional and contained in 
$$
     v^{\epsilon} L^2\Omega^*(M\setminus \pa M,\widetilde{g}_{\QFB})
$$
for some $\epsilon>0$.   
\label{qfb.12d}\end{corollary}

\section{Weighted $L^2$-cohomology of $\QFB$-metrics} \label{wl2.0}

 Let $(M,\phi)$ be a manifold with fibered corners.  If $H_1,\ldots, H_{\ell}$ are the boundary hypersurfaces of $M$, let $x_1,\ldots, x_{\ell}$ be corresponding boundary defining functions compatible with $\phi$  in the sense of \cite{CDR}.  Suppose also that the labelling of the boundary hypersurfaces is compatible with the partial order in the sense that
$$
     H_i< H_j \; \Longrightarrow \; i<j.
$$ 
On $(M,\phi)$, let $g_{\QFB}$ be a $\QFB$-metric with respect to the boundary defining functions $x_1,\ldots, x_{\ell}$.  
\begin{definition}
  A \textbf{quasi-fibered cusp metric} ($\QFC$-metric for short) is a Riemannian metric on $M\setminus \pa M$ of the form 
  \begin{equation}
   g_{\QFC}:= v^2 g
  \label{wl2.2}\end{equation}
  for some $\QFB$-metric $g$, where $v:= \prod_{i=1}^{\ell} x_i$ is the total boundary defining function of $(M,\phi)$.  When $g$ is in fact a $\QAC$-metric, such a metric is also called a \textbf{quasi-asymptotically cylindrical metric} ($\QCyl$-metric for short).
\label{wl2.1}\end{definition}

\begin{example}
When $M$ is a manifold with boundary, a $\QFC$-metric is just a fibered cusp metric in the sense of \cite{HHM2004}, while a $\QCyl$-metric is just a $b$-metric in the sense of \cite{MelroseAPS}.
\label{wl2.3}\end{example}

In a sense, the notion of $\QFC$-metric can be seen as a generalization of the notion of fibered cusp metrics to manifolds with fibered corners of arbitrary depth.  Notice however that $\QFC$-metrics differ fundamentally from the notion of iterated fibered cusp metrics of \cite{DLR,HR}, which are yet another way of generalizing fibered cusps metrics to manifolds with fibered corners of higher depth.  Similarly, $\QCyl$-metrics and the $\Qb$-metrics introduced in \cite{CDR} can both be seen as a generalization of the notion of $b$-metrics (or asymptotically cylindrical metrics) to certain manifolds with fibered  corners of higher depth, though in a different way.  Indeed, as we will see, $\QCyl$-metrics correspond more to a geometric generalization, at least in terms of $L^2$-cohomology, while $\Qb$-metrics correspond more to an analytic generalization in the sense that  it is for this type of metrics that the b-calculus of Melrose is  generalized in \cite{KR1}.  

Before discussing the weighted $L^2$-cohomology of $\QFB$-metrics and $\QAC$-metrics, let us first recall what is the weighted $L^2$-cohomology of a complete Riemannian manifold $(X,g)$ with weight $w\in \CI(X)$ a positive function.  In terms of the space of weighted $L^2$-forms of degree $q$
\begin{equation}
  wL^2\Omega^q(X,g)= \{\omega \; \mbox{a mesurable section of}\; \Lambda^q(T^*X) \; | \; \int_X \| w^{-1} \omega\|^2_g dg <\infty\},
\label{wl2.4}\end{equation}
the weighted $L^2$-cohomology group of degree $q$ is the quotient
\begin{equation}
  \WH^q(X,g,w)=  \frac{\{\omega \in wL^2\Omega^q(X,g)\; | \; d\omega=0\}}{\{d\eta \; | \; \eta\in wL^2\Omega^{q-1}(X,g) \; \mbox{such that} \; d\eta\in wL^2\Omega^q(X,g)\}}.
\label{wl2.5}\end{equation}  
If $w=1$, then this is just the $L^2$-cohomology group of degree $q$
\begin{equation}
  H^q_{(2)}(X,g)= \frac{\{\omega\in L^2\Omega^q(X,g)\; | \; d\omega=0\}}{\{d\eta \; | \; \eta\in L^2\Omega^{q-1}(X,g) \; \mbox{such that} \; d\eta\in L^2\Omega^q(X,g)\}}.
\label{wl2.6}\end{equation}
Considering the subset of weighted $L^2$-forms
\begin{equation}
   wL^2_d\Omega^q(X,g):= \{ \omega\in wL^2\Omega^q(X,g)\; | \; d\omega\in wL^2\Omega^{q+1}(X,g) \},
\label{wl2.7}\end{equation}
notice that the groups \eqref{wl2.5} correspond to the cohomology groups of the complex
\begin{equation}
\xymatrix{
     \cdots \ar[r]^-{d} & wL^2_d\Omega^q(X,g) \ar[r]^-{d} &  wL^2_d\Omega^{q+1}(X,g) \ar[r]^-{d} & \cdots.
}
\label{wl2.8}\end{equation}
Since the image of the exterior derivative is not necessarily closed, it is often interesting to consider as well the reduced weighted $L^2$-cohomology group of degree $q$
\begin{equation}
     \overline{\WH}^q(X,g,w):= \{\omega\in wL^2\Omega^q(X,g)\; | \; d\omega=0\} /\overline{\{d\eta \; | \; \eta\in wL^2\Omega^{q-1}(X,g) \; \mbox{such that} \; d\eta\in wL^2\omega^q(X,g)\}}.\label{wl2.9}\end{equation}
When $w=1$, this is just the reduced $L^2$-cohomology group of degree $q$
\begin{equation}
 \overline{H}^q_{(2)}(X,g):= \{\omega\in L^2\Omega^q(X,g)\; | \; d\omega=0\} /\overline{\{d\eta \; | \; \eta\in L^2\Omega^{q-1}(X,g) \; \mbox{such that} \; d\eta\in L^2\omega^q(X,g)\}}.
 \label{wl2.8new}\end{equation}  
Now, if $\delta_{g,w}$ is the formal adjoint of $d$ with respect to the inner product on $wL^2\Omega^*(X,g)$, then $\overline{\WH}^q(X,g,w)$ is naturally identified with the space of weighted $L^2$-harmonic forms 
\begin{equation}
  L^2\cH^q(X,g,w):= \{ \omega\in wL^2\Omega^q(X,g) \; | \; d\omega= \delta_{g,w} \omega=0\}.  
\label{wl2.10}\end{equation}
Indeed, the isomorphism 
$$
               \overline{\WH}^q(X,g,w)\cong L^2\cH^q(X,g,w)
$$
follows readily from the Kodaira decomposition \cite[Théorème~24 in \S~32]{deRham}
\begin{equation}
  wL^2\Omega^q(X,g)= L^2\cH^q(X,g,w) \oplus \overline{d \Omega^{q-1}_c(X)} \oplus \overline{\delta_{g,w}\Omega_c^{q+1}(X)},
\label{wl2.10}\end{equation}
where $\Omega_c^q(X)$ denotes the space of compactly supported smooth $q$-forms.  Notice that \eqref{wl2.5}, \eqref{wl2.6} and \eqref{wl2.9} only depend on the quasi-isometry class of $g$.  That is, if $g'$ is another complete metric such that for some positive constant $C$,
$$
         \frac{g}{C}< g' < Cg  \quad \mbox{everywhere on} \; X,
$$
then $g'$ has the same weighted $L^2$-cohomology groups or reduced $L^2$-cohomology groups as $g$.

Now, if we take $X= M\setminus \pa M$ and $g= g_{\QFC}$ a $\QFC$-metric, a natural choice for the weight function is to take 
$$
            w=x^a= x_1^{a_1}\cdots x_{\ell}^{a_{\ell}}
$$
for some $a=(a_1,\ldots, a_{\ell})\in \bbR^{\ell}$ and consider the corresponding weighted $L^2$-cohomology group
\begin{equation}
  \WH^q_{\QFC}(M,\phi,a):= \WH^q(M\setminus \pa M, g_{\QFC}, x^a).
\label{wl2.11}\end{equation}

For a $\QFB$-metric  $g_{\QFB}$, following \cite[(13)]{HHM2004} for the case of fibered boundary metrics, we will consider the weighted $L^2$-cohomology group
\begin{equation}
  \WH^q_{\QFB}(M,\phi,a):= \frac{\{\omega\in x^aL^2\Omega^q(M\setminus \pa M, g_{\QFB}) \; | d\omega=0\}}{\{ d\eta \; | \; \eta\in v^{-1}x^a L^2\Omega^{q-1}(M\setminus \pa M, g_{\QFB}), \; d\eta \in x^aL^2\Omega^q(M,g_{\QFB}) \}},
\label{wl2.12}\end{equation}
where we recall that $v= \prod_{i=1}^{\ell}x_i$ is a total boundary defining function of $M$.  Since by \eqref{wl2.2}, we have that 
\begin{equation}
x^aL^2\Omega^q(M\setminus \pa M,g_{\QFB})= v^{\frac{m}2-q}x^a L^2\Omega^q(M\setminus \pa M, g_{\QFC}),
\label{wl2.13}\end{equation}
where $m=\dim M$, we see that \eqref{wl2.12} can be reformulated in terms of $\QFC$-metrics as follows,
\begin{equation}
  \WH^q_{\QFB}(M,\phi,a)= \WH^q_{\QFC}(M,\phi, a+ \underline{\left(\frac{m}2-q\right)}),
\label{wl2.14}\end{equation}
where 
$$
   a+ \underline{\left(\frac{m}2-q\right)}:= \left( a_1+ \left(\frac{m}2-q\right),\ldots, a_{\ell}+ \left(\frac{m}2-q\right)\right).
 $$

Instead of working with general weighted $L^2$-forms, it will be sometimes convenient to work with weighted $\QFB$-conormal forms, especially for arguments involving integration by parts.  More precisely, consider the space
\begin{multline}
  x^a\cA_{\QFC,2}\Omega^q(M):=  \{  \omega\in x^aL^2\Omega^q(M\setminus \pa M,g_{\QFC}) \; | \;  \\
  \forall k\in \bbN_0, \forall X_1,\ldots X_k\in \cV_{\QFB}(M), \quad \nabla_{X_1}\cdots\nabla_{X_k} \omega\in x^aL^2\Omega^q(M\setminus \pa M, g_{\QFC})\},
\label{wl2.15}\end{multline}
with $\nabla$ a choice of covariant derivative for $\Lambda^q({}^{\QFC}T^*M)\to M$, where
\begin{equation}
    {}^{\QFC}T^* M= v({}^{\QFB}T^*M)
\label{wl2.16}\end{equation}
is the $\QFC$-cotangent bundle.  
\begin{lemma}
The weighted $L^2$-cohomology group \eqref{wl2.11} can be written in terms of weighted $\QFB$-conormal forms as
\begin{equation}
  \WH^q_{\QFC}(M,\phi,a):= \frac{\{\omega\in x^a\cA_{\QFC,2}\Omega^q(M)\; | \; d\omega=0\}}{\{d\eta \; | \; \eta\in x^a \cA_{\QFC,2}\Omega^{q-1}(M), \; d\eta\in x^a\cA_{\QFC,2}\Omega^q(M)\}}.
\label{wl2.18}\end{equation}
\label{wl2.17}\end{lemma}
\begin{proof}
Let $d^*=\delta_{g_{\QFC},0}$ be the formal adjoint of $d$ with respect to some choice of $\QFC$-metric $g_{\QFC}$.  Then the Hodge Laplacian $(d+d^*)^2$ is such that 
$v^2(d+d^*)^2$ and $(d+d^*)^2v^2$ are elliptic $\QFB$-operators of order 2.  Thus, using the small pseudodifferential calculus of \cite{KR1} (or even the calculus of pseudodifferential operators with proper support of \cite{ALN2007}), we know by standard arguments that  there exist operators $Q_1,Q_2\in \Psi^{-2}_{\QFB}(M;\Lambda^*({}^{\QFC}T^*M))$ preserving the form degree such that 
\begin{equation}
(d+d^*)^2v^2Q_1= \Id+ R_1, \quad Q_2 v^2(d+d^*)^2= \Id + R_2,  \quad \mbox{with} \;  R_1,R_2\in \Psi^{-\infty}_{\QFB}(M;\Lambda^*({}^{\QFC}T^*M)).
\label{wl2.19}\end{equation}
Now, to prove the lemma, we need to show the natural map
\begin{equation}
 \frac{\{\omega\in x^a\cA_{\QFC,2}\Omega^q(M)\; | \; d\omega=0\}}{\{d\eta \; | \; \eta\in x^a \cA_{\QFC,2}\Omega^{q-1}(M), \; d\eta\in x^a\cA_{\QFC,2}\Omega^q(M)\}}\longrightarrow  \WH^q_{\QFC}(M,\phi,a)
\label{wl2.19new}\end{equation}
is an isomorphism.  To show it is surjective, we need to show that any class in $\WH^q_{\QFC}(M,\phi,a)$ can be represented by a closed form in $x^a\cA_{\QFC,2}\Omega^q(M)$, while to show it is injective, we need to show that given any form $\omega\in x^a L^2_d\Omega^{q-1}(M\setminus\pa M, g_{\QFC})$ with $d\omega\in x^a\cA_{\QFC,2}\Omega^q(M)$, there exists $\eta\in x^a\cA_{\QFC,2}\Omega^{q-1}(M)$ such that $d\eta=d\omega$.  Clearly, as we allow $q$ to vary, both assertions will follow if we can show that given $\omega\in x^a L^2_d\Omega^q(M\setminus\pa M, g_{\QFC})$ such that $d\omega\in x^a\cA_{\QFC,2}\Omega^{q+1}(M)$, there exists $\eta \in x^a\cA_{\QFC,2}\Omega^{q}(M)$ and $\psi\in x^a L^2_d\Omega^{q-1}(M\setminus\pa M, g_{\QFC}) $ such that
$$
      \eta=\omega+d\psi \quad \mbox{and} \quad d\eta=d\omega.
$$

Thus, let $\omega\in x^a L^2_d\Omega^q(M\setminus\pa M, g_{\QFC})$ be such that $d\omega\in x^a\cA_{\QFC,2}\Omega^{q+1}(M)$.  Applying \eqref{wl2.19} to $\omega$, we see that 
$$
     (dd^*+d^*d)v^2Q_1\omega= \omega+ R_1\omega,
$$
so that 
$$
    \eta:= d^*dv^2Q_1\omega-R_1\omega
$$
differs from $\omega$ by an exact form, namely
$$
    \eta=\omega+d\psi 
$$
with $\psi= -d^*(v^2Q_1\omega)\in x^a L^2_d\Omega^{q-1}(M\setminus\pa M, g_{\QFC})$ by the mapping properties of \cite[Theorem~7.4]{KR1}.
 Hence, since $d\eta=d\omega$, we see that 
\begin{equation}
\begin{aligned}
(d+d^*)\eta= d\omega+ d^*\eta =  d\omega-d^*R_1\omega & \Longrightarrow \quad Q_2v^2(d+d^*)^2\eta= Q_2v^2(d+d^*)(d\omega-d^*R_1\omega), \\
         & \Longrightarrow \quad \eta+ R_2\eta= Q_2v^2d^*d\omega-Q_2v^2dd^*R_1 \omega, \\
         & \Longrightarrow \quad \eta= Q_2v^2d^*d\omega -Q_2v^2dd^*R_1\omega- R_2\eta.  
\end{aligned}
\label{wl2.20}\end{equation}
By our assumption and the mapping properties of $\QFB$-operators \cite[Corollary~7.5]{KR1}, 
$$
Q_2v^2d^*d\omega\in x^a\cA_{\QFC,2}\Omega^q(M).
$$   
Moreover, since $Q_2v^2dd^*R_1$ and $R_2$ are operators in $\Psi^{-\infty}_{\QFB}(M;\Lambda^*({}^{\QFC}T^*M))$, we deduce, using the $L^2$-boundedness of $\QFB$-operators of order zero \cite[Theorem~7.4]{KR1} and the fact  that $\QFB$-operators of order $-\infty$ are stable under the action on the left by $\QFB$-vector fields \cite[Corollary~5.2]{KR1},  that 
$$
   \eta\in x^a \cA_{\QFC,2}\Omega^q(M),  \quad d\eta\in x^a \cA_{\QFC,2}\Omega^{q+1}(M),
$$
completing the proof.
\end{proof}

Let $\widehat{M}_{\phi}$ be the Thom-Mather stratified space associated to $(M,\phi)$.  Recall that this space is essentially obtained by collapsing the fibers of $\phi_i: H_i\to S_i$ onto its image for each $i$.  There is in particular a canonical surjective map 
$$
     c_{\phi}: M\to \widehat{M}_{\phi}
$$ 
sending $H_i$ onto a corresponding closed stratum $\overline{s}_i$ in $\widehat{M}_{\phi}$.  On $\widehat{M}_{\phi}$, we can consider for $a=(a_1,\ldots, a_{\ell})\in \bbR^{\ell}$ the sheaf $x^aL^2_{\QFC}\Omega^q$ associated to the presheaf  $x^a\cL^2_{\QFC}\Omega^q$ which to an open set $\cU\subset \widehat{M}_{\phi}$  associates the space of sections
\begin{equation}
  x^a \cL^2_{\QFC}\Omega^q(\cU):=  x^aL^2\Omega^q(c_{\phi}^{-1}(\cU)\setminus (c^{-1}_{\phi}(\cU)\cap\pa M),g_{\QFC}).
\label{wl2.21}\end{equation}
In particular, notice that the space of sections $x^aL^2_{\QFC}\Omega^q(M\setminus\pa M)$ of the sheaf $x^aL^2_{\QFC}\Omega^q$ over $M\setminus \pa M$ corresponds to $q$-forms on $M\setminus\pa M$ that are locally in $L^2_{\QFC}$, while on $\widehat{M}_{\phi}$,
$$
x^a L^2_{\QFC}\Omega^q(\widehat{M}_{\phi}):=  x^aL^2\Omega^q(M\setminus\pa M),g_{\QFC}).
$$
   Let us consider as well as the subsheaf $x^a L^2_{\QFC,d}\Omega^q$ defined by
\begin{equation}
  x^a L^2_{\QFC,d}\Omega^q(\cU):= \{\omega\in x^a L^2_{\QFC}\Omega^q(\cU) \; | \; d\omega \in x^a L^2_{\QFC}\Omega^{q+1}(\cU)\}.  
\label{wl2.22}\end{equation}
This defines a complex 
\begin{equation}
\xymatrix{
   \cdots\ar[r]^-{d} & x^a L^2_{\QFC,d}\Omega^q(\cU) \ar[r]^-{d} & x^a L^2_{\QFC,d}\Omega^{q+1}(\cU) \ar[r]^-{d} & \cdots }.
\label{wl2.23}\end{equation}
In particular, the weighted $L^2$-cohomology groups $\WH^*_{\QFC}(M,\phi,a)$ correspond to the cohomology groups of this complex when we take $\cU= \widehat{M}_{\phi}$.  However, taking advantage of Lemma~\ref{wl2.17}, we can instead work with $\QFB$-conormal forms.  That is, on $\widehat{M}_{\phi}$, we can consider instead for $a=(a_1,\ldots, a_{\ell})\in \bbR^{\ell}$ the sheaf $x^a\cA_{\QFC,2}\Omega^q$ associated to the presheaf $x^a\cL_{\QFC,2}\Omega^q$ which to to an open set $\cU\subset \widehat{M}_{\phi}$ associates 
\begin{multline}
  x^a \cL_{\QFC,2}\Omega^q(\cU):=  \left\{  \omega\in x^aL^2\Omega^q(c_{\phi}^{-1}(\cU)\setminus (c^{-1}_{\phi}(\cU)\cap\pa M),g_{\QFC}) \; | \right. \\
  \forall k\in \bbN_0, \forall X_1,\ldots X_k\in \cV_{\QFB}(M), \quad \nabla_{X_1}\cdots\nabla_{X_k} \omega\in x^aL^2\Omega^q(c_{\phi}^{-1}(\cU)\setminus (c^{-1}_{\phi}(\cU)\cap\pa M),g_{\QFC})\}.
\label{wl2.21b}\end{multline}
We consider as well the subsheaf $x^a \cA_{\QFC,d}\Omega^q$ defined by
\begin{equation}
  x^a \cA_{\QFC,2,d}\Omega^q(\cU):= \{\omega\in x^a \cA_{\QFC,2}\Omega^q(\cU) \; | \; d\omega \in x^a \cA_{\QFC,2}\Omega^{q+1}(\cU)\}.  
\label{wl2.22b}\end{equation}
This defines a complex 
\begin{equation}
\xymatrix{
   \cdots\ar[r]^-{d} & x^a \cA_{\QFC,2,d}\Omega^q(\cU) \ar[r]^-{d} & x^a \cA_{\QFC,2,d}\Omega^{q+1}(\cU) \ar[r]^-{d} & \cdots }
\label{wl2.23b}\end{equation}
and we denote by 
$$
\WH^q_{\QFC}(\cU,\phi,a)= \frac{\{\omega\in  \cA_{\QFC,2,d}\Omega^q(\cU)\; | \; d\omega=0\}}{\{ d\eta\; | \; \eta\in \cA_{\QFC,2,d}\Omega^{q-1}(\cU) \}}
$$ the corresponding cohomology group in degree $q$.  By Lemma~\ref{wl2.17}, the weighted $L^2$-cohomology groups $\WH^*_{\QFC}(M,\phi,a)$ correspond to the cohomology groups of this complex when we take $\cU= \widehat{M}_{\phi}$.

Now, let $p_i$ be a point in the regular part $s_i$ of $\overline{s}_i$.  Let $\widehat {B}_i\subset s_i$ be an open neighborhood of $p_i$ in $s_i$.  Since the map $c_{\phi}$ canonically identifies $S_i\setminus \pa S_i$ with $s_i$, let $B_i$ be the corresponding open set in $S_i\setminus \pa S_i$.  Taking $\widehat{B}_i$ smaller if needed, we can assume the fiber bundle $\phi_i: H_i\to S_i$ is trivial over $B_i$.  Choosing a tubular neighborhood of $H_i$ in $M$ as in \cite[Lemma~1.10]{CDR}, we see that we can choose an open neighborhood $\cU_i$ of $p_i$ in $\widehat{M}_{\phi}$ such that 
\begin{equation}
  c_{\phi}^{-1}(\cU_i)\cong [0,\delta)_{x_i}\times Z_i\times B_i
\label{wl2.24}\end{equation}  
with $\phi_i$ corresponding to the projection $Z_i\times B_i\to B_i$ under this identification.  

Recall however that, as opposed to the fibered corners metrics of \cite{DLR}, $\QFB$-metrics do not decompose nicely in terms of the factorization \eqref{wl2.24}.  Instead of $x_i$, one needs to use the total boundary defining function $v$  and look at the decomposition induced by its level sets.  From that point of view, $c_{\phi}^{-1}(\cU_i)$ can be seen as a subset of the Cartesian product
$$
    [0,\delta)_{v} \times Z_i\times B_i,
$$
namely
\begin{equation}
c_{\phi}^{-1}(\cU_i)\cong \cV_i\times B_i\quad \mbox{with}\quad \cV_i:= \{ (v,z)\in [0,\delta)_v\times Z_i \; | \; w_i(z)>\frac{v}{\delta}\} \subset [0,\delta)_{v}\times Z_i,
\label{wl2.25}\end{equation}
where $w_i$ is a choice of total boundary defining function for $Z_i$.  Now, on $c^{-1}_{\phi}(\cU)$, a simple example of $\QFB$-metric is given by the restriction to $\cV_i\times B_i$ of the metric 
\begin{equation}
          \frac{dv^2}{v^4} + g_{Z_i}+ \frac{g_{B_i}}{v^2} \quad \mbox{on} \; (0,\delta)_{v}\times Z_i \times B_i ,
\label{wl2.25b}\end{equation}
where $g_{B_i}$ is a choice of Riemannian metric on $B_i$ and $g_{Z_i}$ is a choice of $\QFB$-metric on $Z_i$ with structure of manifold with fibered corners induced from the one on $M$.  Hence a corresponding example of $\QFC$-metric on $c_{\phi}^{-1}(\cU)$ is given by the restriction to $\cV_i\times B_i$ of the metric 
\begin{equation}
     \frac{dv^2}{v^2}+ v^2g_{Z_i} + g_{B_i}\quad \mbox{on} \; (0,\delta)_{v}\times Z_i \times B_i.  
\label{wl2.25c}\end{equation}
By the Ehresmann lemma of \cite[Corollary~A.6]{KR3}, notice that a $\QFC$-metric is always locally quasi-isometric to a metric of the form \eqref{wl2.25c}.
To compute the weighted $L^2$-cohomology in these local models, let us first assume that $B_i$ is a point.  One then needs to compute the weighted $L^2$-cohomology of the metric 
\begin{equation}
     \frac{dv^2}{v^2}+ v^2 g_{Z_i}
\label{wl2.26d}\end{equation}
in the region  
$$
     \cV_i=\{  (v,z)\in [0,\delta)_{v}\times Z_i \; | \; w_i(z)>\frac{v}{\delta} \}.
$$
Now, the exterior differential may be decomposed in terms of the decomposition $(0,\delta)\times Z_i$, namely
\begin{equation}
     d= d_{Z_i}+ dv\wedge \frac{\pa}{\pa v}= d_{Z_i}+ \frac{dv}{v}\wedge\left(  v\frac{\pa}{\pa v}\right). 
\label{wl2.26c}\end{equation}
\begin{lemma}
The decomposition \eqref{wl2.26c} holds on $L^2$.
\label{dl2.1}\end{lemma}
\begin{proof}
Since $\cV_i$ is not quite a Cartesian product, we cannot directly apply \cite[Proposition~2.28]{Zucker}.  However, $\cV_i$ sits inside the Cartesian product $[0,\delta)_{v}\times Z_i$.  Moreover, by \cite[Proposition~2.27]{Zucker}, the operator $d$ on $L^2$ is the closure of the exterior derivative of effectively compactly supported smooth forms (in the sense of \cite[p.179]{Zucker}).  But those forms can be extended smoothly to $[0,\delta)_v\times Z_i$.  This means that we can run the approximation argument of the proof of \cite[Proposition~2.28]{Zucker} on $[0,\delta)_v\times Z_i$ and then restrict to $\cV_i$ to obtain the desired result.   
\end{proof}
   On the other hand, on the interval $[0,\delta)$ with $b$-metric $\frac{dv^2}{v^2}$, imposing absolute boundary conditions, namely Neumann boundary conditions for $0$-forms and Dirichlet boundary conditions for $1$-forms at $v=\delta$, we see from \cite{MelroseAPS} that the exterior derivative induces a surjective Fredholm map
\begin{equation}
    \frac{dv}{v}\wedge v\frac{\pa}{\pa v}:  v^{\lambda}H^1_b([0,\delta))\to v^{\lambda} L^2_b([0,\delta); \Lambda^1({}^bT^*([0,\delta))))
\label{wl2.27}\end{equation}
for $\lambda\ne 0$, where $H^1_b([0,\delta))$ is the $b$-Sobolev space of order $1$.  Moreover, its kernel is trivial for $\lambda>0$ and consists of constant functions for $\lambda<0$.  Of course, when $\lambda>0$, the inverse map $G_{\lambda,\delta}$ can be written explicitly,
$$
      G_{\lambda,\delta}(\omega)(v)= \int_0^{v} f(t) dt \quad \mbox{if} \;  \omega(v)= f(v)dv.
$$
For $\lambda<0$, we can consider an inverse $G_{\lambda,\delta}$ on the orthogonal complement of the kernel with respect to the natural inner product on $v^{\lambda}L^2_b([0,\delta))$, 
$$
       \langle f,g\rangle_{\lambda}= \int_{0}^{\delta} v^{-2\lambda}f(v) g(v) \frac{dv}v.
$$
Thus, defining the map 
$$
    \widetilde{G}_{\lambda,\delta}(\omega)(v)= \int_{\delta}^v f(t) dt, \quad \mbox{if} \; \omega(v)=f(v)dv,
$$
we have that 
\begin{equation}
G_{\lambda,\delta}= (\Id-P_{\lambda,\delta})\widetilde{G}_{\lambda,\delta},
\label{fa.1}\end{equation}
where $P_{\lambda,\delta}: v^{\lambda}L^2_b([0,\delta))\to v^{\lambda} L^2_b([0,\delta))$ is the orthogonal projection onto the constant explicitly given by
\begin{equation}
  P_{\lambda,\delta}(f):= -\frac{2\lambda}{\delta^{-2\lambda}}\int_{0}^{\delta} t^{-2\lambda}f(t) \frac{dt}{t}.
\label{fa.2}\end{equation}

Using that the dilation 
$$
       \begin{array}{llcl}
       \mathcal{D}: & [0,1] &\to & [0,\delta] \\ 
          & v & \mapsto & \delta v
       \end{array}
$$
induces a map $\mathcal{D}^*: v^{\lambda}L^2_b([0,\delta)) \to v^{\lambda} L^2_b([0,1))$ such that $\delta^{-\lambda}\mathcal{D}^*$ is an isometry, a simple computation shows that the operator norms of the inverse map $G_{\lambda,\delta}$ and the projection $P_{\lambda,\delta}$ do not depend on $\delta$.  

Now, let $a=(a_1,\ldots, a_{\ell})\in \bbR^{\ell}$ be a choice of weight and let $\omega\in x^a L^2_{\QFC,d}\Omega^q(\cU)$ be a closed form.  In particular, we have that
\begin{equation}
\int_{Z_i}\left(\int_{0}^{\delta w_i(z)} x^{-2a} \|\omega \|_{g_{\QFC}}^2 v^{m_i} \frac{dv}{v}\right) dg_{Z_i} <\infty, 
\label{wl2.28}\end{equation}
where $m_i:=\dim Z_i$ and $w_i$ is a total boundary defining function of $Z_i$ as in \eqref{wl2.25}.
The form $\omega$ decomposes as
\begin{equation}
  \omega= \omega_1 + \frac{dv}v\wedge \omega_2,
\label{wl2.29}\end{equation}
where $\omega_1$ and $\omega_2$ are forms involving no factor $dv$, so that in terms of \eqref{wl2.26c},
\begin{equation}
    d\omega= d_{Z_i}\omega_1+ \frac{dv}v\wedge \left( v\frac{\pa \omega_1}{\pa v}- d_{Z_i}\omega_2 \right), \quad d\omega=0 \; \Longrightarrow \; d_{Z_i}\omega_1=0, \; v\frac{\pa \omega_1}{\pa v}= d_{Z_i}\omega_2.
\label{wl2.29b}\end{equation}
  Thus, if we denote by $h_{\QFC}= w_i^2g_{Z_i}$ the $\QFC$-metric on $Z_i$ associated to the $\QFB$-metric $g_{Z_i}$, then in terms of the decomposition \eqref{wl2.29}, the condition \eqref{wl2.28} means that 
\begin{equation}
\begin{gathered}
  \int_{Z_i} \int_0^{\delta w_i(z)} x^{-2a} \left( \frac{v}{w_i} \right)^{-2q} \|\omega_1(v,z)\|^2_{h_{\QFC}} v^{m_i} \frac{dv}v dg_{Z_i}<\infty,\\
  \int_{Z_i} \int_0^{\delta w_i(z)} x^{-2a}\left( \frac{v}{w_i} \right)^{-2(q-1)}\|\omega_2(v,z)\|^2_{h_{\QFC}} v^{m_i} \frac{dv}{v} dg_{Z_i}<\infty.
\end{gathered}  
\label{wl2.30}\end{equation}
In particular, for each fixed $z\in Z_i$, provided $\lambda=a_i+ (q-1)-\frac{m_i}2$ is non-zero, we can consider the inverse $G_{\lambda,\epsilon}$ with $\epsilon=\delta w_i(z)$ to construct a  map 
\begin{equation}
         \begin{array}{llcl}
         G^q:  & x^aL^2_{\QFC,d}\Omega^q(\cV_i) &\to  &x^aL^2_{\QFC}\Omega^{q-1}(\cV_i) \\
                  & \omega_1+\frac{dv}v\wedge \omega_2 & \mapsto & G_{a_i+(q-1)-\frac{m_i}2,\delta w_i(z)}(\frac{dv}v \wedge \omega_2(v,z)).        
                  \end{array}
\label{wl2.38b}\end{equation}
Using the fact that in $\cV_i$, $\QFB$-vector fields are generated by $v^2\frac{\pa}{\pa v}$ and $\QFB$-vector fields on $Z_i$, we see that this also induces a map
\begin{equation}
         \begin{array}{llcl}
         G^q:  & x^a\cA_{\QFC,2,d}\Omega^q(\cV_i) &\to  &x^a\cA_{\QFC,2}\Omega^{q-1}(\cV_i) \\
                  & \omega_1+\frac{dv}v\wedge \omega_2 & \mapsto & G_{a_i+(q-1)-\frac{m_i}2,\delta w_i(z)}(\frac{dv}v \wedge \omega_2(v,z)).        
                  \end{array}
\label{wl2.38}\end{equation}
By construction, notice that $G^q(\omega)$ has no $dv$ factor and is such that $\frac{dv}v\wedge v\frac{\pa}{\pa v}G^q(\omega)=\frac{dv}v\wedge \omega_2$.  
When $\lambda>0$, this immediately yields the following vanishing result.  
\begin{lemma}
If $a=(a_1,\ldots, a_{\ell})\in \bbR^{\ell}$ is a choice of weight, then for $q-1>\frac{m_i}2-a_i$, we have that
$$
    \WH_{\QFC}^q(\cV_i,\phi,a)=\{0\}.
$$
\label{sa.1}\end{lemma}
\begin{proof}
Let $\omega\in x^a \cA_{\QFC}\Omega^q(\cV_i)$ be a closed form.   Since   we are assuming that $\lambda=a_i+ (q-1)-\frac{m_i}2>0$, we compute that
\begin{equation}
dG^q(\omega) = \int_{0}^v d_{Z_i}\omega_2(t) \frac{dt}t+ \frac{dv}v\wedge \omega_2 = \int_{0}^v  \frac{\pa \omega_1(t)}{\pa t} dt + \frac{dv}v\wedge \omega_2 =\omega_1 +\frac{dv}v\wedge \omega_2= \omega,
\label{fa.2b}\end{equation}
which shows that $\omega$ is necessarily exact in this case.  \end{proof}

If instead  $\lambda=a_i+(q-1)-\frac{m_i}2<0$, then for each fixed $z\in Z_i$, $G^q(\omega)$ is orthogonal to the constant sections of 
$$
[0,\delta w_i(z))\times \Lambda^{q-1}T_z^*Z_i\to [0,\delta w_i(z))
$$
with respect to the weighted inner product
$$
    \langle \eta, \psi\rangle_z= \int_{0}^{\delta w_i(z)}  x^{-2a}\left( \frac{v}{w_i(z)} \right)^{-2(q-1)}\langle \eta, \psi\rangle_{h_{\QFC}} v^{m_i}\frac{dv}v.
$$
In this case, denote by $P_{\lambda,z}$ the projection
\begin{equation}
P_{\lambda,z}(\eta)  :=-\frac{2\lambda}{(\delta w_i(z))^{-2\lambda}} \int_0^{\delta w_i} v^{-2\lambda} \eta(v) \frac{dv}{v}
\label{wl2.40b}\end{equation}
 on the constant section for fixed $z$ induced by \eqref{fa.2} and let 
\begin{equation}
         \begin{array}{llcl}
         P_{q-1}:  & x^a\cA_{\QFC}\Omega^{q-1}(\cV_i) &\to  &x^a\cA_{\QFC}\Omega^{q-1}(\cV_i) \\
                  & \eta_1+\frac{dv}v\wedge \eta_2 & \mapsto & P_{\lambda,z}(\eta_1)
                  \end{array}
\label{wl2.40}\end{equation}
be the induced map as $z$ is allowed to vary.   With this notation understood, we have that
$$
      G^{q}= (\Id-P_{q-1})\widetilde{G}^q
$$
with $\widetilde{G}^q$ the operator defined by
$$
     \widetilde{G}^{q}(\omega)= \int_{\delta w_i(z)}^{v} \omega_2(t) \frac{dt}{t}.
$$
The computation of $dG^q(\omega)$ then yields the following result.
\begin{lemma}
If $a=(a_1,\ldots, a_{\ell})\in \bbR^{\ell}$ is a choice of weight, then for $q<\frac{m_i}2-a_i$, we have that
$$
    \WH_{\QFC}^q(\cV_i,\phi,a) \cong \WH_{\QFC}^q(Z_i,\phi,(a_{i+1},\ldots,a_{\ell})).
$$
\label{sa.2}\end{lemma}
\begin{proof}
Let $\omega\in x^a \cA_{\QFC,d}\Omega^q(\cV_i)$ be a form such that $d\omega$ has no $dv$ factor, so that by \eqref{wl2.29b},
\begin{equation}
     v\frac{\pa \omega_1}{\pa v}= d_{Z_i}\omega_2.
\label{sa.2b}\end{equation} 
In particular, this holds if $\omega\in x^a \cA_{\QFC,d}\Omega^q(\cV_i)$ is a closed form.  
To compute $d G^q(\omega)$, we need to perform a preliminary computation.  First, for $P_{q-1}$, notice that for $\lambda_i:= a_i+(q-1)-\frac{m_i}2<0$ and
$\eta=\widetilde{G}^q(\omega)$, we have that
\begin{equation}
\begin{aligned}
 d_{Z_i}P_{q-1}(\eta) &= d_{Z_i} \left( -\frac{2\lambda_i}{(\delta w_i(z))^{-2\lambda_i}} \int_0^{\delta w_i} v^{-2\lambda_i} \eta(v) \frac{dv}{v} \right) \\
 &= \frac{dw_i}{w_i} \wedge\left( 2\lambda_i P_{q-1}(\eta)-2\lambda_i\eta(\delta w_i)\right) + P_{\lambda_i,z}(d_{Z_i}\eta) \\
 &=\frac{dw_i}{w_i} \wedge\left( 2\lambda_i P_{q-1}(\eta)\right) + P_{\lambda_i,z}(d_{Z_i}\eta), \end{aligned}
\label{fa.3}\end{equation}
where $P_{\lambda_i,z}$ was defined in \eqref{wl2.40b} and in the last line, we have used the fact that $\eta(\delta w_i)= \widetilde{G}^q(\omega)(\delta w_i)=0$.  Hence, we compute that
\begin{equation}
d(G^q(\omega))= d_{Z_i}(\Id-P_{q-1})\widetilde{G}^q(\omega)+ \frac{dv}v\wedge \omega_2= -\frac{dw_i}{w_i}\wedge (2\lambda_iP_{q-1}(\widetilde{G}^q(\omega))) + (\Id-P_{\lambda_i,z})d_{Z_i}\widetilde{G}^q(\omega)+ \frac{dv}v\wedge\omega_2.
\label{fa.4}\end{equation}
But using that $(\Id-P_{\lambda_i,z})$ gives zero when acting on sections constant in $v$, we see that 
\begin{equation}
\begin{aligned}
(\Id-P_{\lambda_i,z})d_{Z_i}\widetilde{G}^q(\omega)&= (\Id-P_{\lambda_i,z}) d_{Z_i}\left( \int_{\delta w_i}^{v} \omega_2(t) \frac{dt}{t} \right) \\
 &= (\Id-P_{\lambda_i,z}) \left( -\frac{dw_i}{w_i}\wedge \omega_2(\delta w_i)+ \int_{\delta w_i}^v d_{Z_i}\omega_2(t) \frac{dt}{t} \right) \\
 &= (\Id-P_{\lambda_i,z}) \int_{\delta w_i}^v\frac{\pa}{\pa t}\omega_1(t) dt= (\Id-P_{\lambda_i,z}) (\omega_1(v)-\omega_1(\delta w_i))\\
 &= (\Id-P_{\lambda_i,z}) \omega_1(v).
  \end{aligned}
\label{fa.5}\end{equation}
Hence, we see that 
\begin{equation}
  d(G^q(\omega))= -\frac{dw_i}{w_i}\wedge \left( 2\lambda_i P_{q-1}(\widetilde{G}^q(\omega)) \right) + (\Id-P_{\lambda_i,z})\omega_1 + \frac{dv}{v}\wedge \omega_2.
\label{fa.6}\end{equation}
Since $\frac{dw_i}{w_i}$ has bounded pointwise norm with respect to the metric $h_{\QFC}$, we see  that 
$$
\frac{dw_i}{w_i}\wedge \left( 2\lambda_i P_{q-1}(\widetilde{G}^q(\omega) \right)\in x^a\cA_{\QFC,2}\Omega^q(\cV_i).
$$
On the other hand, with respect to the pointwise norm $|\cdot |_{h_{\QFC}}$ induced by the metric $h_{\QFC}$, we have that 
\begin{equation}
\begin{aligned}
|P_{\lambda_i,z}\omega_1|_{h_{\QFC}} &= -\frac{2\lambda_i}{(\delta w_i)^{-2\lambda_i}}\left|\int_0^{\delta w_i}t^{-2\lambda_i}\omega_1(t)\frac{dt}t\right|_{h_{\QFC}} \\
   &= -\frac{2\lambda_i}{(\delta w_i)^{-2\lambda_i}}\left|\int_0^{\delta w_i}t^{-2\lambda_i-2}t^2\omega_1(t)\frac{dt}t\right|_{h_{\QFC}} \\ 
   &\le -\frac{2\lambda_i}{(\delta w_i)^{-2\lambda_i-2}}\left|\int_0^{\delta w_i}t^{-2\lambda_i-2}\omega_1(t)\frac{dt}t\right| =\frac{-2\lambda_i}{-2\lambda_i-2}|P_{q}(\omega_1)|_{h_{\QFC}},\\
\end{aligned}
\label{fa.7}\end{equation}
which shows that $P_{\lambda_i,z}\omega_1\in x^a\cA_{\QFC,2}\Omega^q(\cU_i)$.  Hence, this shows that $dG^q(\omega)\in x^a\cA_{\QFC,2}\Omega^{q}(\cV_i)$, that is, that
$$
     G^q(\omega)\in x^a\cA_{\QFC,2,d}\Omega^{q-1}(\cV_i).
$$

Finally this means that $\omega$ and $\omega'=\omega-d G^q(\omega)$ are such that $d\omega=d\omega'$, and by construction, $\omega'$ has no $dv$ factor.  In particular, if $\omega$ is closed, then $\omega$ and $\omega'$ represent the same cohomology class with $\omega'$ having no $dv$ factor.  Since it is closed, it must be constant in $v$.  Clearly moreover, if $\omega'= d\eta$ with $\eta\in x^a\cA_{\QFC,2}\Omega^{q-1}(\cV_i)$, then $d\eta$ has no $dv$ factor, so by the same argument, $\eta'=\eta-dG^{q-1}(\eta)$ has no $dv$ factor and is such that $d\eta'=\omega'$.  Since $\omega'$ has no $dv$ factor, this means that $\eta'$ must also be constant in $v$.  This means that to compute weighted $L^2$-cohomology in $\cV_i$ for $\QFB$-conormal forms, we only need to use weighted  conormal forms in $\cV_i$ that have no $dv$ factor and are constant in $v$.  
Now, in $c^{-1}_{\phi}(\cU_i)$, we can suppose that $x_1=\cdots= x_{i-1}=1$, so we can write 
\begin{equation}
     x^a= v^{a_i}x_{i+1}^{a_{i+1}-a_i}\cdots x_{\ell}^{a_{\ell}-a_i}= v^{a_i}w_i^{-a_i}x_{i+1}^{a_{i+1}}\cdots x_{\ell}^{a_{\ell}}.
\label{wl2.31}\end{equation}
Thus, using the condition \eqref{wl2.30} with $\omega_1= \omega'$ and $\omega_2=0$, we find using \eqref{wl2.31} and integrating in $v$ that 
\begin{equation}
\|\omega'\|^2_{x^aL^2_{\QFC}\Omega^q(\cU)}=\left( \frac{\delta^{m_i-2q-2a_i}}{m_i-2q-2a_i} \right)\int_{Z_i} \|\omega' \|^2_{h_{\QFC}} x_{i+1}^{-2a_{i+1}}\cdots x_{\ell}^{-2a_{\ell}} dh_{\QFC}<\infty.
\label{wl2.32}\end{equation}  
Hence, these weighted conormal forms constant in $v$ are in one-to-one correspondence with 
\begin{equation}
     x_{i+1}^{a_{i+1}}\cdots x_{\ell}^{a_{\ell}} \cA_{\QFC,2,d}\Omega^q(Z_i) \quad \mbox{for} \; q<\frac{m_i}2-a_i,
\label{wl2.33}\end{equation}
from which the result follows.  
\end{proof}

Let us now consider the case when $q-1<\frac{m_i}{2}-a_i<q$.  In this case, we can get some vanishing result provided some extra condition is satisfied.
\begin{lemma}
Let $a=(a_1,\ldots, a_{\ell})\in \bbR^{\ell}$ be a weight such that $q-1<\frac{m_i}2-a_i\le q$.  If  the map $\WH_{\QFC}^q(Z_i,\phi,(a_{i+1},\ldots,a_{\ell}))\to H^q(Z_i)$ is injective, then
$$
    \WH_{\QFC}^q(\cV_i,\phi,a) =\{0\}.
$$\label{sa.3}\end{lemma}
\begin{proof}
Let $\omega\in x^a \cA_{\QFC,2}\Omega^q(\cV_i)$ be a closed form.  In this setting, it is easier to consider $\widetilde{G}^q(\omega)$ instead $G^q(\omega)$.  Then one computes that 
\begin{equation}
\begin{aligned}
d\widetilde{G}^q(\omega)&= d \left( \int_{\delta w_i(z)}^{v} \omega_2(t) \frac{dt}t \right) \\
          &= \frac{dv}{v}\wedge \omega_2(v)- \frac{dw_i}{w_i}\wedge\omega_2(\delta w_i) + \int_{\delta w_i(z)}^v d_{Z_i} \omega_2(t) \frac{dt}t \\
          &=\frac{dv}{v}\wedge \omega_2(v)- \frac{dw_i}{w_i}\wedge\omega_2(\delta w_i) + \int_{\delta w_i(z)}^v \frac{\pa \omega_1}{\pa t}(t)  dt \\
          &=\omega(v) - \left(  \omega|_{v=\delta w_i}\right).
\end{aligned}
\label{cs.1}\end{equation}
Since $\omega|_{v=\delta w_i}$ is constant in $v$ and $q-1<\frac{m_i}{2}-a_i\le q$, the only way it can be an element of $x^a\cA_{\QFC,2}\Omega^q(\cV_i)$ is if it vanishes trivially.  If this is not the case, then $\widetilde{G}^q(\omega)$ is an element of $x^a\cA_{\QFC,2}\Omega^{q-1}(\cV_i)$, but not of $x^a\cA_{\QFC,2,d}\Omega^{q-1}(\cV_i)$.  However,  thanks to the fact that $q-1<\frac{m_i}2-a_i$, we can try to look for a form $\eta\in x^a\cA_{\QFC,2}\Omega^{q-1}(\cV_i)$ constant in $v$ such that 
$$
     d\eta= d_{Z_i}\eta= \omega|_{v=\delta w_i}.
$$
Indeed, if such a form exists, the result follows by replacing $\widetilde{G}^q(\omega)$ with $\widetilde{G}^q(\omega)+\eta$, since then
$$
   \omega= d(\widetilde{G}^q(\omega)+\eta) \quad \mbox{and} \quad  \widetilde{G}^q(\omega)+\eta\in x^a\cA_{\QFC,2,d}\Omega^{q-1}(\cV_i).
$$
Clearly, such a form $\eta$ exists if and only if the form $\omega|_{v=\delta w_i}\in x_{i+1}^{a_{i+1}}\cdots x_{\ell}^{a_{\ell}}\cA_{\QFC,2}\Omega^q(Z_i)$ defines a trivial cohomology class in $\WH_{\QFC}^q(Z_i,\phi,(a_{i+1},\ldots,a_{\ell}))$, so that the proof is completed by the lemma below and our assumption that the natural map
 $\WH^q_{\QFC}(Z_i,\phi, (a_{i+1},\ldots,a_{\ell}))\to H^q(Z_i)$ is an inclusion.
\end{proof}

\begin{lemma}
If $\omega\in x^a\cA_{\QFC,2}\Omega^q(\cV_i)$ is a closed form of degree $q$ such that $q-1<\frac{m_i}{2}-a_i\le q$, then $\omega|_{v=\delta w_i}$ defines a trivial cohomology class in $H^q(Z_i)$.
\label{cs.2}\end{lemma}
\begin{proof}
Instead of the coordinates $(v,z)$, we will use the coordinates $(x_i,z)$ in $\cV_i$, so that $v=\delta w_i$ corresponds to $x_i=\delta$.  Write now $\omega$ in terms of this decomposition,
$$
     \omega= \omega_1 + dx_1\wedge \omega_2,
$$
where this time the $\omega_j$ do not involve $dx_1$  and are distinct from those occurring in \eqref{wl2.29}.  In terms of this decomposition,
\begin{equation}
 0=d\omega = d_{Z_i}\omega_1 + dx_i\wedge\left(  \frac{\pa \omega_1}{\pa x_i}- d_{Z_i}\omega_2\right) \; \Longrightarrow \; d_{Z_i}\omega_1=0, \quad \frac{\pa\omega_1}{\pa x_i}=d_{Z_i}\omega_2,
\label{sa.4}\end{equation}
so we see that the cohomology class of $\omega_1\in H^q(Z_i)$ does not depend on $x_i$.   To see that this class must in fact vanish, let us restrict $\omega_1$ to the region $C_{i}$ of $\cV_i$ where $w_i\ge \epsilon$ for some $\epsilon>0$.  Thus,
$$
     C_{i}\cong [0,1)_{x_i} \times K_{i},
$$
where $K_{i}\subset Z_i$ is the compact region of $Z_i$ where $w_i\ge \epsilon$.  In the region $C_{i}$, the $\QFC$-metric \eqref{wl2.26d} is quasi-isometric to the cusp metric
\begin{equation}
   \frac{dx_i^2}{x_i^2}+ x_i^2 g_{K_i},
\label{cs.4}\end{equation}
where $g_{K_i}$ is some (geodesically incomplete) Riemannian metric on the manifold with boundary $K_i$.  By the discussion above, $\omega_1(x_i)$ defines a class in $H^q(K_{i})$ which does not depend on $x_i$.  Putting absolute boundary condition on the boundary $\pa K_i$, let $\nu$ be the harmonic representative of $\omega_1(x_i)$ on $K_i$.  In particular, $\nu$ does not depend on $x_i$ and 
\begin{equation}
     \|\omega_1(x_i)\|_{L^2,g_{K_i}} \ge \|\nu\|_{L^2,g_{K_i}}  \quad \forall x_i,
\label{cs.5}\end{equation}
where $\|\cdot\|_{L^2,g_{K_i}}$ is the $L^2$-norm induced by the metric $g_{K_i}$.  However, since   $\omega \in x^aL^2_{\QFC}\Omega^q(\cV_i)$, when we restrict to $C_i$, we see that
\begin{equation}
    \int_0^{\delta} \|\omega_1(x_i)\|^2_{L^2,g_{K_i}}  x_i^{-2\alpha_i} \frac{dx_i}{x_i}<\infty \quad \mbox{with} \; \alpha_i= a_i+q-\frac{m_i}2\ge 0.
\label{cs.6}\end{equation}
The combination of \eqref{cs.5} and \eqref{cs.6} implies that $\nu\equiv 0$, that is, that the cohomology class of $\omega_1(x_1)$ in $H^q(Z_i)$ must vanish.  
 \end{proof}

Combining all these lemmas yields the following local description of weighted $L^2$-cohomology of a $\QFC$-metric.

\begin{theorem}
Let  $a=(a_1,\ldots, a_{\ell})\in \bbR^{\ell}$ be such that  $a_i\ne \frac{m_i}2-q+1$ for each $q\in \{0,1,\ldots, m_i\}$.  Suppose moreover that for $q$ such that $q-1<\frac{m_i}2-a_i\le q$,  the natural map $\WH_{\QFC}^q(Z_i,\phi, (a_{i+1},\ldots,a_{\ell}))\to H^q(Z_i)$ is an inclusion.  Then for an open set $\cU_i$ as in \eqref{wl2.24} with $B_i$ contractible with smooth boundary, we have that
\begin{equation}
  \WH_{\QFC}^q(\cU_i,\phi,a)= \left\{  \begin{array}{ll} 
        \WH_{\QFC}^q(Z_i,\phi,(a_{i+1},\ldots,a_{\ell})), & \mbox{if} \; q<\frac{m_i}2-a_i; \\
        \{0\}, & \mbox{otherwise}.
  \end{array}  \right.
\label{wl2.26b}\end{equation}
\label{wl2.26}\end{theorem}
\begin{proof}
Since the model $\QFC$-metric \eqref{wl2.25c} corresponds to the Cartesian product of $(\cV_i, \frac{dv^2}{v^2}+ v^2 g_{Z_i})$ with $(B_i, g_{B_i})$ and since the range of the exterior differential on $B_i$ is closed for absolute boundary conditions, we can use \cite[Theorem~2.29]{Zucker} to conclude that 
$$
     \WH_{\QFC}^q(\cU_i,\phi,a)= \WH_{\QFC}^q(\cV_i,\phi,a)   
$$ 
using the fact that $B_i$ is contractible.  Hence, the result follows from Lemma~\ref{sa.1}, Lemma~\ref{sa.2} and Lemma~\ref{sa.3}. 
\end{proof}
This has various consequences.
\begin{corollary}
Let  $a=(a_1,\ldots, a_{\ell})\in \bbR^{\ell}$ be such that  for each $i$ and each $q\in \{0,1,\ldots, m_i\}$, $a_i\ne \frac{m_i}2-q+1$.  Suppose moreover that for each $i$ and $q$ such that $q-1<\frac{m_i}2-a_i\le q$,  the natural map $\WH_{\QFC}^q(Z_i,\phi, (a_{i+1},\ldots,a_{\ell}))\to H^q(Z_i)$ is an inclusion.  Then the cohomology groups $\WH_{\QFC}^q(M,\phi,a)$ are finite dimensional for each $q$.  
\label{sa.6}\end{corollary}
\begin{proof}
When $Z_i$ is a closed manifold, the weighted $L^2$-cohomology is identified with de Rham cohomology, which is finite dimensional.  Hence proceeding by induction on the depth, we can show using Mayer-Vietoris long exact sequences in weighted $L^2$-cohomology that $\WH_{\QFC}^*(Z_i,\phi,a)$ is finite dimensional for each $i$.  The same argument then shows that $\WH_{\QFC}^*(M,\phi,a)$ is finite dimensional as well.  
\end{proof}
Notice in particular that this corollary implies that the operator
\begin{equation}
 d: x^aL^2_{\QFC,d}\Omega^q(M\setminus \pa M)\to x^aL^2_{\QFC}\Omega^{q+1}(M\setminus \pa M)
\label{sa.7}\end{equation}
has closed range.  Assuming slightly more yields the following.  
\begin{corollary}
For a weight $a\in \bbR^{\ell}$ such that both $a$ and $-a$ satisfy the hypotheses of Corollary~\ref{sa.6}, the operator \eqref{sa.7}
has closed range, as well as it formal adjoint $d^*_a= x^{2a}\circ d^* \circ x^{-2a}$ with respect to the inner product on $x^a L^2_{\QFC}\Omega^*(M\setminus \pa M)$.  Moreover, the self-adjoint extension of the Hodge-deRham operator $\eth_a:= d+ d^*_a$ on $x^aL^2_{\QFC}\Omega^*(M\setminus \pa M)$ is Fredholm.  Finally, there is a natural identification 
\begin{equation}
\WH_{\QFC}^q(M,\phi,\pm a)\cong L^2\cH^q(M\setminus \pa M,g_{\QFC},x^{\pm a})
\label{sa.8b}\end{equation}
and the natural pairing  
$$
     (\omega,\eta)\to \int_{M\setminus \pa M} \omega\wedge \eta
$$
between $x^aL^2\Omega^q(M\setminus M,g)$ and $x^{-a}L^2\Omega^{m-q}(M\setminus M,g)$ induces a non-degenerate pairing (Poincaré duality) between $\WH^q(M,\phi, a)$ and $\WH^{m-q}(M,\phi, -a)$.
\label{sa.8}\end{corollary}
\begin{proof}
As already noted, the operator \eqref{sa.7} has closed range.  Since $d^*= \pm * d *$ where $*$ is the Hodge star operator, we see also that $d^*_a$ has closed range from the fact that the operator 
$$
d: x^{-a}L^2_{\QFC,d}\Omega^q(M\setminus \pa M)\to x^{-a}L^2_{\QFC}\Omega^{q+1}(M\setminus \pa M)
$$
has closed range.  Hence, the Kodaira decomposition \eqref{wl2.10} induces the natural identification \eqref{sa.8b}.  
Since $d$ and $d^*_a$ have orthogonal images in $x^aL^2_{\QFC}\Omega^*(M\setminus \pa M)$, the domain of $\eth_a$ is the intersection of the domains of $d$ and $d^*_a$.  We also see from the Kodaira decomposition that $\eth_a$ must have closed range.  Moreover, its kernel corresponds to $L^2\cH^*(M\setminus\pa M, g_{\QFC},x^a)$, which is finite dimensional by \eqref{sa.8b}.  The same is true for its cokernel by formal self-adjointness, so that $\eth_a$ must be Fredholm as claimed.  Finally, a quick computation shows that the operator $*_{a}:=x^{-2a}*$ induces an isomorphism
$$
    *_a: L^2\cH^q(M\setminus\pa M,g_{\QFC},x^a)\to L^2\cH^{m-q}(M\setminus\pa M,g_{\QFC},x^{-a})
$$
with inverse $(-1)^{q(m-q)}x^{2a}*$, establishing Poincaré duality between $\WH^q(M,\phi, a)$ and $\WH^{m-q}(M,\phi, -a)$.
\end{proof}

For more specific choices of weights, these weighted $L^2$-cohomology groups can be identified with intersection cohomology.  
\begin{corollary}
Suppose that $\dim Z_{i}>0$ for each $i$, so $\widehat{M}_{\phi}$ has no singular stratum of codimension $1$.  Let $\mathfrak{p}$ be a perversity and let $a=(a_1,\ldots,a_{\ell})$ be the weight defined by $a_i= \mathfrak{p}(m_i+1)-\frac{m_i-2}{2}-\epsilon$ for some fixed $0<\epsilon<\frac12$.  Suppose that for each $i$ and $q$ such that $q-1<\frac{m_i}2-a_i\le q$, the natural map 
$$
\WH^q_{\QFC}(Z_i,\phi, (a_{i+1},\ldots,a_{\ell}))\to H^q(Z_i)
$$ 
is an inclusion.  Then there is an isomorphism
$$
        \WH^*_{\QFC}(M,\phi,a)\cong \IH^*_{\mathfrak{p}}(\widehat{M}_{\phi}).
$$
\label{wl2.35}\end{corollary}
\begin{proof}
With this choice of weight, the statement of Theorem~\ref{wl2.26} can be reformulated as 
\begin{equation}
  \WH^q(\cU_i,\phi,a)= \left\{  \begin{array}{ll} 
        \WH^q(Z_i,\phi,(a_{i+1},\ldots, a_{\ell})), & \mbox{if} \; q<m_i-1-\mathfrak{p}(m_i+1)+\epsilon; \\
        \{0\}, & \mbox{otherwise}.
  \end{array}  \right.
\label{wl2.36}\end{equation}
This is the same behavior as $\IH^*_{\mathfrak{p}}(\cU_i)$, \cf \cite[(8)]{HHM2004}.  Thus, proceeding by induction on the depth of $(M,\phi)$, we can use \cite[Proposition~1]{HHM2004} with \eqref{wl2.36} to obtain the result.  

\end{proof}

Note that in this corollary, the upper middle perversity $\overline{\mathfrak{m}}$ defined by $\overline{\mathfrak{m}}(\ell)=\lfloor \frac{\ell-1}2\rfloor$ (also) corresponds to choosing the weight $a=(\epsilon,\ldots, \epsilon)$, while lower middle perversity $\underline{\mathfrak{m}}$ defined by  $\underline{\mathfrak{m}}(\ell)=\lfloor \frac{\ell-2}2\rfloor$ (also) corresponds to the weight $a=(-\epsilon,\ldots, -\epsilon)$.  In this case, with some further assumptions on the dimensions of the fibers of the fiber bundles $\phi_i: H_i\to S_i$, we can also compute $L^2$-cohomology.

\begin{corollary}
Suppose that for all $i\in \{1,\ldots,\ell\}$, $\dim Z_{i}$ is odd and the natural map 
$$\WH^{\frac{m_i+1}2}_{\QFC}(Z_i,\phi, 0)\to H^{\frac{m_i+1}2}(Z_i)$$
 is an inclusion.  Then
\begin{equation}
      L^2\cH^*(M\setminus \pa M, g_{\QFC})\cong \WH^*_{\QFC}(M,\phi,0)\cong \IH^*_{\overline{\mathfrak{m}}}(\widehat{M}_{\phi})= \IH^*_{\underline{\mathfrak{m}}}(\widehat{M}_{\phi}).      
\label{l2qfc.1b}\end{equation}      
\label{l2qfc.1}\end{corollary}
\begin{proof}
First, since $\dim Z_i$ is odd for all $i$, the stratified space $\widehat{M}_{\phi}$ is Witt, so there is a natural identification
$$
\IH^*_{\overline{\mathfrak{m}}}(\widehat{M}_{\phi})= \IH^*_{\underline{\mathfrak{m}}}(\widehat{M}_{\phi}).
$$
More importantly, the hypothesis that $\dim Z_i$ is odd and that $\WH^{\frac{m_i+1}2}_{\QFC}(Z_i,\phi, 0)\to H^{\frac{m_i+1}2}(Z_i)$ is an inclusion 
for all $i$ allows us to take $a=(0,\ldots,0)$ in Theorem~\ref{wl2.26}, Corollary~\ref{sa.6} and Corollary~\ref{sa.8}, giving in particular the identification 
$$
   L^2\cH^*(M\setminus \pa M, g_{\QFC})\cong \WH^*_{\QFC}(M,\phi,0).
$$
Finally, to obtain the identification with upper middle intersection cohomology, notice that \eqref{wl2.26b} gives in this case
\begin{equation}
\WH^q(\cU_i,\phi,0)= \left\{  \begin{array}{ll} 
        \WH^q(Z_i,\phi,0), & \mbox{if} \; q<\frac{m_i}{2}; \\
        \{0\}, & \mbox{otherwise},
  \end{array}  \right.
\label{l2qfc.2}\end{equation}
which, since $m_i$ is odd, is the same as the local behavior of upper middle and lower middle intersection cohomology \cite[(8)]{HHM2004}.  The result therefore follows from \cite[Proposition~1]{HHM2004}.
\end{proof}
\begin{corollary}
If $\dim Z_i$ is odd and $\IH^{\frac{m_i+1}2}_{\underline{\mathfrak{m}}}(\widehat{(Z_i)}_{\phi})=\{0\}$ for all $i$, then 
$$
L^2\cH^*(M\setminus \pa M, g_{\QFC})\cong \WH^*_{\QFC}(M,\phi,0)\cong \IH^*_{\overline{\mathfrak{m}}}(\widehat{M}_{\phi})= \IH^*_{\underline{\mathfrak{m}}}(\widehat{M}_{\phi}).
$$
\label{wl2.41}\end{corollary}
\begin{proof}
In this case, one can check recursively with respect to the depth of $\widehat{M}_{\phi}$ that the hypotheses of Corollary~\ref{l2qfc.1} are satisfied, so that \eqref{l2qfc.1b} holds.
\end{proof}
\begin{remark}
In \cite{HR}, a related result was obtained for iterated fibered cusp metrics.  
\label{wl2.37}\end{remark}

\section{$L^2$-cohomology of the Nakajima metric} \label{hs.0}

We can apply the results of the previous section to study the weighted $L^2$-cohomology of the Nakajima metric $g_n$ on the Hilbert scheme  or Douady space $\Hilb_0^n(\bbC^2)$ of $n$ points on $\bbC^2$.  Recall that $\Hilb_0^n(\bbC^2)$ is a crepant resolution
\begin{equation}
  \pi: \Hilb_0^n(\bbC^2)\to (\bbC^2)^n_0/\mathbf{S}_n
\label{hs.1}\end{equation}  
of the quotient of 
$$
   (\bbC^2)^n_0= \{q\in (\bbC^2)^n\; | \; \sum_j q_j=0\}
$$ 
by the action of the symmetric group $\mathbf{S}_n$ given by
$$
     \sigma\in \mathbf{S}_n, \; q\in (\bbC^2)^n_0, \quad \sigma\cdot q:= (q_{\sigma^{-1}(1)},\ldots, q_{\sigma^{-1}(n)}).
$$
For instance, when $n=2$, $(\bbC^2)^2_0=\bbC^2$ and $\Hilb_0^2(\bbC^2)=T^*\bbC\bbP^1$ is just the standard crepant resolution $T^*\bbC\bbP^1\to \bbC^2/\bbZ_2$ obtained by blowing up the origin of $\bbC^2/\bbZ_2$.  

By the work of Nakajima \cite{Nakajima}, the Hilbert scheme $\Hilb_0^n(\bbC^2)$ admits a natural complete hyperKähler metric $g_n$ obtained via the hyperKähler quotient construction of \cite{HKLR1987}.  In \cite{Joyce}, Joyce introduced the notion of $\QALE$-metrics and constructed a hyperKähler example on $\Hilb_0^n(\bbC^2)$ by solving an appropriate complex Monge-Ampère equation.  It was subsequently shown by Carron \cite{Carron2011} that these two hyperKähler metrics coincide, showing in particular that the Nakajima metric is a $\QALE$-metric.     

In particular, $(\Hilb_0^n(\bbC^2), g_n)$ admits a $\QAC$-compactification.  To describe it, let us review from \cite{Carron2011} how the crepant resolution \eqref{hs.1} can be performed iteratively using the notion of local product resolution of \cite{Joyce}.  Let $\mf{p}=(I_1,\ldots, I_k)$ be a partition of $\{1,\ldots,n\}$ and consider the associated vector space
\begin{equation}
 V_{\mf{p}}= \{ q\in (\bbC^2)^n_0 \; | \; \forall \ell\in \{1,\ldots,k\},  \; q_i=q_j   \;\forall i,j\in I_{\ell}\}.
\label{hs.2}\end{equation} 
There are corresponding subgroups
\begin{equation}
\begin{gathered}
    A_{\mf{p}}= \{\gamma\in \mathbf{S}_n\; | \; \gamma \cdot  q=q  \; \forall q\in V_{\mf{p}}\}\cong \mathbf{S}_{n_1}\times \cdots\times \mathbf{S}_{n_k}, \\
    B_{\mf{p}}= \{\gamma\in \mathbf{S}_n\; | \; \gamma \cdot  V_{\mf{p}}= V_{\mf{p}}\}.
\end{gathered}    
\label{hs.3}\end{equation}
Clearly, $A_{\mf{p}}$ is a normal subgroup of $B_{\mf{p}}$, so the quotient $N_{\mf{p}}:= B_{\mf{p}}/A_{\mf{p}}$ is a group as well.  
\begin{example}
For $i,j\in \{1,\ldots,n\}$ two distincts elements, consider the partition
\begin{equation}
      \mf{p}_{i,j}= \{\{i,j\}, \{k_1\},\ldots, \{k_{n-2}\}\},
\label{hs.4b}\end{equation}
where $\{k_1,\ldots, k_{n-2}\}=\{1,\ldots,n\}\setminus \{i,j\}$.  In this case, $V_{i,j}:=V_{\mf{p}_{i,j}}= \{q\in (\bbC^2)^n_0\; | \; q_i=q_j\}$, $A_{i,j}:=A_{\mf{p}_{i,j}}\cong \mathbf{S}_2\cong \bbZ_2$, $B_{i,j}:=B_{\mf{p}_{i,j}}= \mathbf{S}_2\times \mathbf{S}_{n-2}$ and $N_{i,j}=N_{\mf{p}_{i,j}}= \mathbf{S}_{n-2}$.  In fact, $B_{i,j}= A_{i,j}\times N_{i,j}$ is just a product.  In general however, $B_{\mf{p}}$ is only a semi-direct product of $A_{\mf{p}}$ and $N_{\mf{p}}$.  For instance,  when $n=4$ and $\mf{p}=\{\{1,2\},\{3,4\}\}$, then $A_{\mf{p}}=\mathbf{S}_2\times \mathbf{S}_2=\bbZ_2\times \bbZ_2$,
$N_{\mf{p}}=\bbZ_2$ acts by swapping or not the two clusters $\{1,2\}$ and $\{3,4\}$ of $\mf{p}$, and $B_{\mf{p}} =A_{\mf{p}}\rtimes N_{\mf{p}}$ is only a semi-direct product.   
\label{hs.4}\end{example}

Let also 
$
W_{\mf{p}}= V_{\mf{p}}^{\perp} \cong \bigoplus^{k}_{\ell=1}(\bbC^2)^{n_{\ell}}_0
$
be the orthogonal orthogonal complement of $V_{\mf{p}}$, where $n_{\ell}=|I_{\ell}|$.  Now, $A_{\mf{p}}$ acts on $W_{\mf{p}}$ and the quotient $W_{\mf{p}}/A_{\mf{p}}= \bigoplus_{\ell=1}^k (\bbC^2)^{n_{\ell}}/\mathbf{S}_{n_{\ell}}$ admits a natural crepant resolution, namely
\begin{equation}
\pi_{\mf{p}}: \Hilb^{\mf{p}}_0(\bbC^2):= \prod_{\ell}^{k} \Hilb^{n_{\ell}}_0(\bbC^2)\to W_{\mf{p}}/A_{\mf{p}}.  
\label{hs.5}\end{equation}
On the other hand, the action of $B_{\mf{p}}$ on $W_{\mf{p}}\times V_{\mf{p}}=(\bbC^2)^n_0$ descends to an action of $N_{\mf{p}}=B_{\mf{p}}/A_{\mf{p}}$ on $(W_{\mf{p}}\times V_{\mf{p}})/A_{\mf{p}}= (W_{\mf{p}}/A_{\mf{p}})\times V_{\mf{p}}$, which in turns lift to an action of $N_{\mf{p}}$ on $\Hilb^{\mf{p}}_0(\bbC^2)\times V_{\mf{p}}$.   In fact, $N_{\mf{p}}$ acts on each factor of  $\Hilb^{\mf{p}}_0(\bbC^2)\times V_{\mf{p}}$ separately, so the quotient $\left(\Hilb^{\mf{p}}_0(\bbC^2)\times V_{\mf{p}}\right)/N_{\mf{p}}$ is the total space of a flat  orbibundle 
\begin{equation}
\xymatrix{
     \Hilb^{\mf{p}}_0(\bbC^2) \ar[r] &  \left(\Hilb^{\mf{p}}_0(\bbC^2)\times V_{\mf{p}}\right)/N_{\mf{p}}\ar[d] \\
       &   V_{\mf{p}}/N_{\mf{p}}.
}
\label{hs.5b}\end{equation}

\begin{example}
For $\mf{p}=\mf{p}_{i,j}$, the action of $N_{\mf{p}}\cong \mathbf{S}_{n-2}$ on $\Hilb^{\mf{p}_{i,j}}_0(\bbC^2)\times V_{\mf{p}_{i,j}}$ is trivial on the first factor, while on the second factor is the action of $\mathbf{S}_{n-2}$ on the variables complementary to $q_i$ and $q_j$ in $V_{\mf{p}_{i,j}}$.  The corresponding orbibundle in \eqref{hs.5b} is therefore trivial.  If instead we take $n=4$ and $\mf{p}=\{\{1,2\},\{3,4\}\}$, then the corresponding action of $N_{\mf{p}}\cong \bbZ_2$ on $\Hilb^{\mf{p}}_0(\bbC^2)=\Hilb^2_0(\bbC^2)^2= (T^*\bbC\bbP^1)^2$ is generated by the involution
$$
     \Hilb^2_0(\bbC^2)^2\ni (z_1,z_2) \mapsto (z_2,z_1)\in \Hilb^2_0(\bbC^2)^2,
$$ 
so the flat orbibundle \eqref{hs.5b} has some non-trivial holonomy and is therefore non-trivial.
\label{hs.6}\end{example}

To see how these resolutions sit inside $\Hilb^n_0(\bbC^2)$, recall that there is a partial order\footnote{We use the convention opposite to the one of \cite{Carron2011} to be consistent with the partial order of the boundary hypersurfaces of the $\QAC$-compactification of $\Hilb^n_0(\bbC^2)$; see \eqref{hs.9b}.} on partitions given by
\begin{equation}
  \mf{p}\le \mf{q} \; \Longleftrightarrow \; V_{\mf{p}}\le V_{\mf{q}}.
\label{hs.6b}\end{equation}
  In other words, $\mf{p}\le \mf{q}$ if and only if $\mf{q}$ is a refinement of $\mf{p}$.  With respect to this partial order, there is a unique maximal partition $\mf{p}_{\infty}:=\{\{1\},\ldots, \{n\}\}$, as well as a unique minimal partition $\mf{p}_0=\{\{1,\ldots,n\}\}$.  Notice that the partitions $\mf{p}_{i,j}$ of \eqref{hs.4b} are precisely those that are just below the maximal one.  In fact, for $\mf{p}\ne \mf{p}_{\infty}$, we have that 
$$
      V_{\mf{p}}= \bigcap_{\mf{p}\le \mf{p}_{i,j}} V_{i,j}.
$$ 
 
Now, set 
$$
  \Delta_{\mf{p}}= \{ (i,j)\in \{1,\ldots,n\}^2 \; | \; \mf{p}\nleq \mf{p}_{i,j}\}
$$
and consider the set
$$
           \Sigma_{\mf{p}}:=\left(\bigcup_{(i,j)\in \Delta_{\mf{p}}} V_{i,j}\right)/A_{\mf{p}}.
$$
For $R>0$, consider the neighborhood 
$$
      T_{\mf{p}}= \{q\in (\bbC^2)^n_0 \; | \; \exists (i,j)\in \Delta_{\mf{p}}, \; |q_i-q_j|<R\}/A_{\mf{p}}
$$
of $\Sigma_{\mf{p}}$.    With respect to the resolution \eqref{hs.5}, we can then consider the neighborhood
$$
   \cU_{\mf{p}}=\left( (\pi_{\mf{p}}\times \Id)^{-1}(T_{\mf{p}}) \right)/N_{\mf{p}}.
$$
Then there are local biholomorphisms 
\begin{equation}
\nu_{\mf{p}}: \left(\left( (\bbC^2)^n_0/A_{\mf{p}} \right)\setminus T_{\mf{p}}\right)/N_{\mf{p}}\to (\bbC^2)^n_0/\mathbf{S}_n
\label{hs.7}\end{equation}
and 
\begin{equation}
\mu_{\mf{p}}: \left( \left( \Hilb^{\mf{p}}_0(\bbC^2)\times V_{\mf{p}}  \right)/N_{\mf{p}} \right)\setminus \cU_{\mf{p}}\to \Hilb^n_0(\bbC^2)
\label{hs.8}\end{equation}
inducing the commutative diagram
\begin{equation}
\xymatrix{
      \left( \left( \Hilb^{\mf{p}}_0(\bbC^2)\times V_{\mf{p}}  \right)/N_{\mf{p}} \right)\setminus \cU_{\mf{p}} \ar[r]^-{\mu_{\mf{p}}} \ar[d]^{\pi_{\mf{p}}\times \Id} & \Hilb^{n}_0(\bbC^2)\ar[d]^{\pi} \\
  \left(\left( (\bbC^2)^n_0/A_{\mf{p}} \right)\setminus T_{\mf{p}}\right)/N_{\mf{p}}\ar[r]^-{\nu_p} &  (\bbC^2)^n_0/\mathbf{S}_n.  
}
\label{hs.9}\end{equation}
These commutative diagrams precisely indicate how the resolution $\pi: \Hilb^n_0(\bbC^2)\to (\bbC^2)^n_0/\mathbf{S}_n$ can be decomposed in terms of local product resolutions, at least away from the origin.

There is of course a natural action of $\mathbf{S}_n$ on partitions, namely, for $\sigma \in \mathbf{S}_n$, 
$$
    \sigma\cdot \mf{p}= \mf{q} \;  \Longleftrightarrow \; \sigma\cdot V_{\mf{p}}= V_{\mf{q}}.
$$  
This induces the equivalence relation
$$
    \mf{p} \sim \mf{q}  \; \Longrightarrow \; \exists \sigma\in \mathbf{S}_n, \; \sigma\cdot \mf{p}= \mf{q}.
$$
In other words, if $\chi: (\bbC^2)^n_0\to (\bbC^2)^n_0/\mathbf{S}_n$ is the quotient map, then $\mf{p}\sim\mf{q}$ if and only if $\chi(V_{\mf{p}})=\chi(V_{\mf{q}})$.  Under this equivalence relation, an equivalence class corresponds to a partition of $n$ indistinguishable points.  

Denoting by $(\tHilb^n_0(\bbC^2),\phi_n)$ the $\QAC$-compactification of $(\Hilb^n_0(\bbC^2),g_n)$, we see from \eqref{hs.9} that each equivalence class $[\mf{p}]$ of a non minimal partition $\mf{p}$ corresponds to a boundary hypersurface $H_{[\mf{p}]}$ of $\tHilb^n_0(\bbC^2)$.  In fact, $[\mf{p}]\mapsto H_{[\mf{p}]}$ gives a one-to-one correspondence between (non minimal) partitions of $n$ indistinguishable points and the boundary hypersurfaces of $\tHilb^n_0(\bbC^2)$.  Moreover, the partial order on the boundary hypersurfaces of $\tHilb^n_0(\bbC^2)$ is the one induced by the one on partitions,
\begin{equation}
     H_{[\mf{p}]}\le H_{\mf{[q]}}\; \Longleftrightarrow \; [\mf{p}]\le [\mf{q}] \; \Longleftrightarrow \; \exists \mf{p}\in [\mf{p}], \mf{q}\in [\mf{q}] \; \mbox{such that} \; \mf{p}\le \mf{q}.
\label{hs.9b}\end{equation}

To describe the associated fiber bundle on $H_{[\mf{p}]}$ coming from the iterated fibration structure, notice that the vector space $V_{\mf{p}}$ can also be seen as a cone with cross-section $L_{\mf{p}}$ a sphere of real dimension $2\dim_{\bbC}V_{\mf{p}}-1$.  The action of $N_{\mf{p}}$ induces one on this cross-section, so that 
$V_{\mf{p}}/N_{\mf{p}}$ can be regarded as a cone with cross-section $\overline{s}_{\mf{p}}:= L_{\mf{p}}/N_{\mf{p}}$.  This cross-section is singular, but it has at most orbifold singularities, so it is a stratified space admitting a resolution by a manifold with fibered corners that we will denote $S_{\mf{p}}$.  This is precisely the base of the fiber bundle 
\begin{equation}
  \phi_{[\mf{p}]}: H_{[\mf{p}]}\to S_{[\mf{p}]},
\label{hs.10}\end{equation}
whose fiber is then the $\QAC$-compactification $\tHilb^{\mf{p}}_0(\bbC^2)$ of $\Hilb^{\mf{p}}_0(\bbC^2)= \Hilb^{n_{1}}_0(\bbC^2)\times\cdots\times \Hilb^{n_{k}}_0(\bbC^2)$ equipped with the $\QALE$-metric given by the product of Nakajima metrics $g_{\mf{p}}:=g_{n_1}\times g_{n_2}\times \cdots\times g_{n_k}$.  Indeed, by \cite{KR3}, $g_{\mf{p}}$ is again a $\QALE$-metric and its $\QAC$-compactification  $\tHilb^{\mf{p}}_0(\bbC^2)$ can be obtained from the Cartesian product
$$
 \tHilb^{n_{1}}_0(\bbC^2)\times\cdots\times \tHilb^{n_{k}}_0(\bbC^2)
$$
by blowing-up (in the sense of Melrose \cite{MelroseAPS}) certain corners.  Let us denote by $\phi^{\mf{p}}$ the corresponding iterated fibration structure.  By \cite{KR3}, for each boundary hypersurface of $\tHilb^{\mf{p}}_0(\bbC^2)$, the fibers of the induced fiber bundle is of the form $\tHilb^{\mf{q}}_0(\bbC^2)$ for some refined partition $\mf{q}\ge \mf{p}$.

We are now ready to apply Theorem~\ref{wl2.26} to the $\QAC$-compactification $(\tHilb^{n}_0(\bbC^2),\phi_n)$.  Since the corresponding $\QFC$-metric is in fact a $\QCyl$-metric, we will use the notation
$$
\WH^*_{\QCyl}(\tHilb^n_0(\bbC^2),\phi_n,a):= \WH^*_{\QFC}(\tHilb^n_0(\bbC^2),\phi_n,a).
$$

\begin{corollary}
  Fix $0<\epsilon<\frac12$ and let $\mf{p}$ be a partition of $\{1,\ldots,n\}$.  If $a$ is the multiweight equal to $\epsilon$ for each boundary hypersurfaces of $\tHilb^{\mf{p}}_0(\bbC^2)$, then 
\begin{equation}
  \WH^*_{\QCyl}(\tHilb^{\mf{p}}_0(\bbC^2),\phi^{\mf{p}},a)\cong H^*_c(\operatorname{Hilb}_0^{\mf{p}}(\bbC^2)) \quad \mbox{and} \quad 
  \WH^*_{\QCyl}(\tHilb^{\mf{p}}_0(\bbC^2),\phi^{\mf{p}},-a)\cong H^*(\operatorname{Hilb}_0^{\mf{p}}(\bbC^2)). 
\label{sa.10c}\end{equation} 
In particular, if $\mf{p}=\mf{p}_0=\{\{1,\ldots,n\}\}$, this gives  
\begin{equation}
  \WH^*_{\QCyl}(\tHilb^n_0(\bbC^2),\phi_n,a)\cong H^*_c(\operatorname{Hilb}_0^n(\bbC^2)) \quad \mbox{and} \quad 
  \WH^*_{\QCyl}(\tHilb^n_0(\bbC^2),\phi_n,-a)\cong H^*(\operatorname{Hilb}_0^n(\bbC^2)). 
\label{sa.10b}\end{equation}
\label{sa.10}\end{corollary}
\begin{proof}
For this choice of multiweight $a$, the statement of Theorem~\ref{wl2.26} can be reformulated as 
\begin{equation}
\WH_{\QCyl}^q(\cU_i,\phi,-a)= \left\{  \begin{array}{ll} 
        \WH_{\QCyl}^q(Z_i,\phi,(-a_{i+1},\ldots,-a_{\ell})), & \mbox{if} \; q\le \frac{m_i}2; \\
        \{0\}, & \mbox{otherwise}.
  \end{array}  \right.
\label{sa.11}\end{equation}
Now, by \cite[Corollary~5.10]{Nakajima}, we know that $\operatorname{Hilb}_0^n(\bbC^2)$ has no cohomology above the middle degree for all $n$.  By Künneth formula, the same is true for $\Hilb^{\mf{p}}_0(\bbC^2)$ for all partitions $\mf{p}$.  Since for each $i$,  $Z_i$ is $\tHilb^{\mf{q}}_0(\bbC^2)$ for some refined partition $\mf{q}\ge \mf{p}$,   we see proceeding by induction on $\dim\Hilb^{\mf{p}}_0(\bbC^2)$ that we can assume already that 
$$
       \WH^q(Z_i,\phi,(-a_{i+1},\ldots,-a_{\ell}))= H^q(Z_i),
$$ 
so that Theorem~\ref{wl2.26} can be applied and implies that 
$$
     \WH_{\QCyl}^q(\cU_i,\phi,-a)= H^q(Z_i)    
$$
since $Z_i$ has no cohomology above the middle degree.  
Hence, on the stratified space associated to 
$$
(\tHilb^{\mf{p}}_0(\bbC^2),\phi^{\mf{p}}),
$$ 
the weighted $L^2$-cohomology associated to the $\QCyl$-metric and the multiweight $-a$ has the same local behavior as absolute cohomology on $\Hilb^n_0(\bbC^2)$.  Using the five-lemma and commutative diagrams of Mayer-Vietoris long exact sequences, one can therefore show that the natural map 
\begin{equation}
   \WH^*_{\QCyl}(\tHilb^{\mf{p}}_0(\bbC^2),\phi^{\mf{p}},-a)\to H^*(\operatorname{Hilb}_0^{\mf{p}}(\bbC^2))
 \label{sa.12}\end{equation}
is an isomorphism.  For the multiweight $a$, using that the map $H_c^{\frac{m_i}2}(Z_i)\to H^{\frac{m_i}2}(Z_i)$ is an isomorphism, one can dualize the argument, or more simply use Poincaré duality to see that the map 
\begin{equation}
   H_c^*(\operatorname{Hilb}_0^{\mf{p}}(\bbC^2))\to  \WH^*_{\QCyl}(\tHilb^{\mf{p}}_0(\bbC^2),\phi^{\mf{p}},a)
\label{sa.13}\end{equation}
dual to \eqref{sa.12} is also an isomorphism.  
\end{proof}

This can be used to compute the $L^2$-cohomology of $(\Hilb^n_0(\bbC^2),g_n)$.  For this, we need to invoke the decay of $L^2$-harmonic forms of \cite{KR1}. 

\begin{proposition}
For $n\ge 2$, there exists $\epsilon>0$ and a $\QALE$-metric $\hat{g}_n$ quasi-isometric to $g_n$ such that the space of $L^2$-harmonic forms on $(\Hilb^n_0(\bbC^2),\hat{g}_n)$ is finite dimensional and included in $v^{\epsilon}L^2\Omega^*(\Hilb^n_0(\bbC^2),\hat{g}_n)$, where $v$ is a total boundary defining function for $\tHilb^n_0(\bbC^2)$. 
\label{nak.6}\end{proposition}
\begin{proof}
This is a consequence Corollary~\ref{qfb.12d}.  First, by the vanishing theorem of Hitchin \cite{Hitchin}, the fibers of \eqref{hs.10} have only non-trivial $L^2$-cohomology in middle degree.  As discussed, for any partition $\mathfrak{p}>\mathfrak{p}_0$, the stratified space $\overline{s}_{\mathfrak{p}}= L_{\mathfrak{p}}/N_{\mathfrak{p}}$ is indeed the quotient of a sphere by a finite group of isometries $N_{\mathfrak{p}}$.  Furthermore,
$$
       \dim_{\bbR} S_{\mathfrak{p}}= \dim_{\bbR} L_{\mathfrak{p}}= 2\dim_{\bbC}V_{\mathfrak{p}}-1\ge 3,
$$
where $\dim_{\bbC}V_{\mathfrak{p}}$ is a positive even integer.  Similarly, for partitions $\mathfrak{p}>\mathfrak{q}>\mathfrak{p}_0$, we have that
$$
       \dim_{\bbR} S_{\mathfrak{p}}-\dim_{\bbR} S_{\mathfrak{q}}-1= 2(\dim_{\bbC}V_{\mathfrak{p}}-\dim_{\bbC}V_{\mathfrak{q}})-1\ge 2(2)-1=3.
$$
Hence, all the hypotheses of Corollary~\ref{qfb.12d} are satisfied.  Applying it yields the result.
\end{proof}

Corollary~\ref{sa.10} and Proposition~\ref{nak.6} can then be combined to give a proof of the Vafa-Witten conjecture \cite{Vafa-Witten}.

\begin{theorem} For all $n\ge 2$,
\begin{equation}
       L^2\cH^*(\Hilb^n_0(\bbC^2),g_n)\cong \Im ( H^*_c(\Hilb^n_0(\bbC^2))\to H^*(\Hilb^n_0(\bbC^2))).
\label{hs.12b}\end{equation}\label{hs.12}\end{theorem}
\begin{proof}
By the vanishing result of Hitchin \cite{Hitchin} and \cite[Corollary~5.10]{Nakajima}, we only need to show that \eqref{hs.12b} holds in middle degree, that is, in degree $2n-2$.  In other words, by the result of \cite[\S~1.3]{Anderson} or \cite[Lemma~1.4]{Segal-Selby}, we need to show that the natural injective map
\begin{equation}
    \Im ( H^*_c(\Hilb^n_0(\bbC^2))\to H^*(\Hilb^n_0(\bbC^2)))\to L^2\cH^*(\Hilb^n_0(\bbC^2),g_n)\cong L^2\cH^*(\Hilb^n_0(\bbC^2),\hat{g}_n)
\label{hs.12c}\end{equation} 
is also surjective.  To see this, it suffices to show by Corollary~\ref{sa.10} and the conformal invariance of the $L^2$-norm of middle degree forms, that the map
\begin{equation}
    \Im (\WH^*_{\QCyl}(\tHilb^n_0(\bbC^2),\phi_n,a)\to \WH^*_{\QCyl}(\tHilb^n_0(\bbC^2),\phi_n,-a) )\to L^2\cH^*(\Hilb^n_0(\bbC^2),\hat{g}_n)
\label{hs.12c}\end{equation} 
is surjective, where $a$ is the multiweight given by $a_i=\epsilon$ for all $i$ with $\epsilon>0$ sufficiently small.  However, by Proposition~\ref{nak.6}, a harmonic form
$\omega\in L^2\cH^*(\Hilb^n_0(\bbC^2),\hat{g}_n)$ defines a class in $(\WH^*_{\QCyl}(\tHilb^n_0(\bbC^2),\phi_n,a)$ and is therefore in the image of \eqref{hs.12c}, showing that the map is surjective.
\end{proof}

\section{$L^2$-cohomology of quasi-asymptotically conical metrics of depth 2}   \label{do.0}

In this section, we consider a $\QAC$-metric $g_{\QAC}$ of depth 2, that is, the corresponding manifold with fibered corners $(M,\phi)$ is of depth 2.  We make the following assumption on $g_{\QAC}$.

\begin{assumption}
For each submaximal boundary hypersurface $H_i$, we suppose that each fiber $Z_i$ of  $\phi_i: H_i\to S_i$   is at least $4$-dimensional with boundary $\pa Z_i$ such that its universal cover $\widetilde{\pa Z_i}$ is a closed manifold which is a rational homology sphere. 
\label{do.1}\end{assumption}

We will make use of the following basic fact about the $L^2$-cohomology of $b$-metrics.

\begin{lemma}
Let $Z$ be a compact manifold with boundary $\pa Z$ and let $x\in \CI(Z)$ be a boundary defining function.  If $g_{b}$ is a $b$-metric on $Z\setminus \pa Z$, then 
\begin{equation}
   \WH^q(Z\setminus \pa Z, g_b, x^a)\cong \left\{ \begin{array}{ll} H^q(Z), & a<0, \\ H^q_c(Z), & a>0.  \end{array} \right.
\label{do.2b}\end{equation}
Moreover, if $\pa Z$ is a rational homology sphere and $a\ne 0$, then the natural map 
\begin{equation}
     \WH^q(Z\setminus\pa Z,g_{b},x^a)\to H^q(Z)
\label{do.2c}\end{equation}
is injective unless $a>0$ and $q=\dim Z$.  
\label{do.2}\end{lemma}
\begin{proof}
The description \eqref{do.2b} is a standard result which is essentially a particular and easier case of \cite[Proposition~2]{HHM2004}.  Indeed, given $p\in \pa Z$, choosing a neighborhood $\cU$ of $p$ of the form $\cU\cong \cB\times [0,\epsilon)$ with $\cB\subset \pa Z$ a contractible neighborhood of $p$ in $\pa Z$, we see applying the Künneth formula of \cite[Corollary~2.34]{Zucker} that for $a\ne 0$, 
$$
\WH^q(\cU\setminus (\cU\cap \pa Z), g_b, x^a)\cong \left\{ \begin{array}{ll} H^0(\cB)\cong \bbR, & a<0, q=0 \\ \{0\}, & \mbox{otherwise}.  \end{array} \right.
$$ 
Thus, for $a<0$, the weighted $L^2$-cohomology behaves locally like absolute cohomology, while for $a>0$, it behaves locally like compactly supported cohomology.  Using the five-lemma and commutative diagrams of Mayer-Vietoris long exact sequences, one can therefore establish \eqref{do.2b}.  

For \eqref{do.2c}, notice that the assumption that $\pa Z$ is a rational homology sphere implies through the long exact sequence in cohomology of the pair $(Z,\pa Z)$ that the natural map
\begin{equation}
   H^q_c(Z)\to H^q(Z) \quad \mbox{is an isomorphism for} \quad 0<q<\dim Z. 
\label{do.3}\end{equation} 
Since $H^0_c(Z)=\{0\}$, the injectivity of \eqref{do.2c}  (unless $a>0$ and $q=\dim Z$) follows from \eqref{do.2b} and \eqref{do.3}.

\end{proof}

This observation allows us to apply the results of \S~\ref{wl2.0} as follows.  

\begin{proposition}
Suppose that $g_{\QAC}$ is a $\QAC$-metric of depth $2$ satisfying Assumption~\ref{do.1}.  Let $a=(a_1,\ldots,a_{\ell})\in \bbR^{\ell}$ be a multiweight such that $a_i\ne \frac{m_i}2-q+1$ for each $q\in\{0,1,\ldots,m_i\}$.  Then the natural map 
\begin{equation}
   \WH^q_{\QCyl}(Z_i\setminus\pa Z_i,\phi, (a_{i+1},\ldots,a_{\ell}))\to H^q(Z_i)
\label{do.4b}\end{equation}
is an inclusion whenever $q-1<\frac{m_i}2-a_i\le q$.  In particular, the conclusions of Theorem~\ref{wl2.26} and Corollaries~\ref{sa.6} and \ref{sa.8} hold for $\WH^*_{\QCyl}(M,\phi,a)$.  
\label{do.4}\end{proposition}
\begin{proof}
By \eqref{do.2c} in the previous lemma, the only way \eqref{do.4b} could fail to be an inclusion is if $q=m_i$, in which case $a_i<1-\frac{m_i}2\le 0$ since we assume $m_i\ge 4$.  Hence, the map \eqref{do.4b} is still an inclusion by Lemma~\ref{do.2}.
\end{proof}

This has the following consequences for the weighted cohomology of $g_{\QAC}$.  
\begin{proposition}
Suppose $g_{\QAC}$ is a $\QAC$-metric of depth $2$ satisfying Assumption~\ref{do.1}.  For $0<\epsilon<\frac12$,  consider the multiweight $a=(\epsilon,\ldots,\epsilon)\in \bbR^{\ell}$.  Then 
$$
    \WH^q_{\QAC}(M,\phi,\pm a)= \WH^q_{\QCyl}(M,\phi,\pm a+ \underline{\left(\frac{m}2-q\right)})
$$
is finite dimensional and $\WH^q_{\QAC}(M,\phi, a)$ is Poincar\'e dual to
$\WH^{m-q}_{\QAC}(M,\phi,-a)$.  More importantly, there are natural maps
\begin{gather}
\label{do.5b}  \WH^q_{\QAC}(M,\phi, a)\to L^2\cH^q(M\setminus \pa M, g_{\QAC}), \\
\label{do.5c} L^2\cH^q(M\setminus \pa M, g_{\QAC})\to \WH^q_{\QAC}(M,\phi,-a),
\end{gather}
with composition
\begin{equation}
\WH^q_{\QAC}(M,\phi, a)\to L^2\cH^q(M\setminus \pa M, g_{\QAC})\to \WH^q_{\QAC}(M,\phi,-a)\label{do.5d}\end{equation}
corresponding to the obvious map.  
 \label{do.5}\end{proposition}
\begin{proof}
Finite dimensionality and Poincar\'e duality is a consequence of Proposition~\ref{do.4}, Corollary~\ref{sa.6} and Corollary~\ref{sa.8}.  
Clearly, there is a map
$$
\{\omega\in x^aL^2\Omega^q(M\setminus\pa M,g_{\QAC})\; |\; d\omega=0\}\to \overline{\WH}^q(M\setminus \pa M,g_{\QAC})\cong  L^2\cH^q(M\setminus\pa M,g_{\QAC}).
$$
To see that it induces a well-defined map \eqref{do.5b}, we need to check that 
$$
  \{d\eta\; | \; \eta\in v^{-1}x^aL^2\Omega^{q-1}(M\setminus\pa M,g_{\QAC}), d\eta \in x^aL^2\Omega^{q}(M\setminus\pa M,g_{\QAC})\}      
$$
maps to zero in $\overline{\WH}^q(M\setminus \pa M,g_{\QAC})$, where $v\in \CI(M)$ is a total boundary defining function inducing the $\QAC$-structure on $(M,\phi)$.  To this end, let $\psi\in \CI(\bbR)$ be a function equal to $1$ on $(-\infty,1]$ and to $0$ on $[2,\infty)$ and consider the sequence
$$
       \eta_k= \psi(-\log v -k)\eta,  \quad \mbox{for} \; k\in \bbN.
$$
Then $\eta_k$ is of compact support and clearly 
$$
   d\eta_k = \psi'(-\log v-k)\left(-\frac{dv}{v^2}  \right)\wedge (v\eta)+ \psi(-\log v-k)d\eta\to d\eta \in L^2(M\setminus\pa M, g_{\QAC})
$$
since $\frac{dv}{v^2}$ is bounded with respect to the norm induced by $g_{\QAC}$ and $v\eta\in x^aL^2(M\setminus\pa M,g_{\QAC})$.  Approximating each $\eta_k$ by a smooth compactly supported form, we thus see that $d\eta$ is in the $L^2$-closure of the image of $d: \Omega_c^{q-1}(M\setminus\pa M)\to \Omega^{q}_c(M\setminus\pa M)$, so vanishes in $\overline{\WH}^q(M\setminus \pa M,g_{\QAC})$.  On the other hand, the map \eqref{do.5c} is just the natural map.  Finally, to see that the composition \eqref{do.5d} is the obvious map, we proceed as in \cite[Lemma~1.4]{Segal-Selby}.  More precisely, if 
$$
   \beta= \lim_{k\to\infty} d\gamma_k \; \mbox{in}\; L^2\Omega^q(M\setminus\pa M, g_{\QAC})
$$
with $\gamma_k\in \Omega^{q-1}_c(M\setminus\pa M)$, then $d\beta=0$.  Moreover, for $\alpha\in x^aL^2\Omega^{m-q}(M\setminus\pa M,g_{\QAC})$ with $d\alpha=0$ representing a class in $\WH^{m-q}_{\QAC}(M,\phi,a)$, we have that
$$
   \int_{M\setminus \pa M} \beta \wedge \alpha = \lim_{k\to \infty} \int_{M\setminus \pa M} d(\gamma_k\wedge \alpha)=0.
 $$
 Hence, by Poincar\'e duality, $\beta\equiv 0$ in $\WH^q_{\QAC}(M,\phi,-a)$.  This shows that the map \eqref{do.5c} can be defined by looking at the image of a representative of a class in $\overline{\WH}^q(M\setminus \pa M,g_{\QAC})$, hence that the composition \eqref{do.5d} is just the obvious map.  
\end{proof}

This can be used to compute the $L^2$-cohomology of $g_{\QAC}$ provided we make the following assumption.
\begin{assumption}
For each boundary hypersurface $H_i$, we suppose that \eqref{qfb.12a} holds $H_j= H_i$ and that \eqref{qfb.12aa} holds for each $H_j<H_i$.   For $H_i$ submaximal, we suppose also that a fiber $Z_i$ of $\phi: H_i\to S_i$ has only non-trivial $L^2$-cohomology in middle degree whenever $\dim Z_i=4$.    
\label{do.5e}\end{assumption}

\begin{theorem}
Let $g_{\QAC}$ be a $\QAC$-metric of depth 2 on $(M,\phi)$ and suppose that Assumption~\ref{do.1} and Assumption~\ref{do.5e} hold.  In this case, the reduced $L^2$-cohomology of $g_{\QAC}$ is given by
\begin{equation}
L^2\cH^q(M\setminus \pa M, g_{\QAC})\cong \Im(\WH^q_{\QAC}(M,\phi,a)\to \WH^q(M,\phi,-a)),
\label{do.6b}\end{equation}
where $a=(\epsilon,\ldots,\epsilon)\in\bbR^{\ell}$ with $0<\epsilon<\frac12$.  
\label{do.6}\end{theorem}  
\begin{proof}
First, notice that by Assumption~\ref{do.1} and Assumption~\ref{do.5e}, Theorem~\ref{qfb.12} holds.  In particular, for $H_i$ submaximal with $\dim Z_i>4$, \eqref{qfb.12b} holds thanks to Assumption~\ref{do.1}, while it holds also appealing to Assumption~\ref{do.5e} when $\dim Z_i=4$.  Hence, changing $g_{\QAC}$ in its quasi-isometry class, we can assume that 
$$
    L^2\cH^q(M\setminus\pa M, g_{\QAC})\subset v^{\epsilon}L^2\Omega^q(M\setminus \pa M, g_{\QAC})
$$
for some small $\epsilon>0$, where $v=\prod_i x_i$ is a total boundary defining function.  This induces a  map  
\begin{equation}
    L^2\cH^q(M\setminus\pa M, g_{\QAC})\to\WH^q(M,\phi,a)
    \label{do.6c} \end{equation}
 and shows that the natural map \eqref{do.5b} is surjective.  This also shows that the natural map \eqref{do.5c} is injective.  Indeed, if $\omega\in L^2\cH^q(M\setminus\pa M,g_{\QAC})$ is non-zero, then 
$$
         \int_{M\setminus\pa M} \omega\wedge *\omega= \|\omega\|^2_{L^2_{\QAC}}\ne 0,
$$
which implies by Poincaré duality that $\omega$ represents a non-zero element in 
$$
       \WH^q_{\QAC}(M,\phi,-a)\cong \left[ \WH^{m-q}_{\QAC}(M,\phi,a) \right]^*.
$$
We are ready to conclude.  The map \eqref{do.6c} induces the map
$$
L^2\cH^q(M\setminus\pa M, g_{\QAC})\to \Im(\WH^q(M,\phi,a)\to \WH^q(M,\phi,-a)),
$$
which by Proposition~\ref{do.5}, the surjectivity of \eqref{do.5b} and the injectivity of \eqref{do.5c} must be a bijection.  
\end{proof}
\begin{remark}
By Proposition~\ref{do.4} and Theorem~\ref{wl2.26}, the cohomology groups 
$$
\WH^q_{\QAC}(M,\phi,\pm a)= \WH^q_{\QCyl}(M,\phi,\pm a+ \underline{\left(\frac{m}2-q\right)})
$$
are computable through Mayer-Vietoris long exact sequences in that they can be described as the Cech cohomology of sheaves on $\widehat{M}_{\phi}$ whose local cohomology admits a local description.  In fact, Corollary~\ref{wl2.35} suggests that they should correspond to some sort of intersection cohomology groups, but  with the caveat that  $\widehat{M}_{\phi}$ has a singular stratum of codimension 1.
\label{do.7}\end{remark}

Our result can be applied in particular to $\QALE$-metrics of depth $2$ on crepant resolutions of $\bbC^n/\Gamma$ with $\Gamma$ a finite subgroup of $\SU(n)$.  

\begin{corollary}
Let $g_{\QALE}$ be a $\QALE$-metric of depth $2$ on a crepant resolution of $\bbC^n/\Gamma$ with $\Gamma$ a finite group of $\SU(n)$.  Denote by $(M,\phi)$ the associated manifold with fibered corners.  For $H_i$ submaximal with $\dim S_i=1$, suppose that $\ker\mathfrak{d}_{S_i}=\{0\}$ where $\mathfrak{d}_{S_i}$ is the  Hodge-deRham operator on $S_i$ associated to the flat bundle of fiberwise $L^2$-harmonic forms on $H_i\to S_i$.  In this case, the reduced $L^2$-cohomology of $g_{\QALE}$ is given by  
\begin{equation}
L^2\cH^q(M\setminus \pa M, g_{\QALE})\cong \Im(\WH^q_{\QAC}(M,\phi,a)\to \WH^q_{\QAC}(M,\phi,-a)),
\label{do.8b}\end{equation}
where $a=(\epsilon,\ldots,\epsilon)\in\bbR^{\ell}$ with $0<\epsilon<\frac12$.  
\label{do.8}\end{corollary}
\begin{proof}
Let us first check that Assumption~\ref{do.1} holds.  Thus, let $H_i$ be submaximal and let $Z_i$ be a fiber of the fiber bundle $\phi_i:H_i\to S_i$.  Since an element  $\gamma\in \SU(n)$ is such that $\dim_{\bbC}\ker(\gamma-\Id)\ge n-1$  if and only if $\gamma=\Id$, we see that $Z_i$ must at least be of complex dimension $2$, that is, at least of real dimension $4$.  On the other hand, since $Z_i\setminus \pa Z_i$ is a crepant resolution of a quotient $\bbC^k/\Gamma$ of $\bbC^k$ by a finite subgroup $\Gamma$ of $\SU(k)$, we see that $\pa Z_i$ is a quotient of sphere, so Assumption~\ref{do.1} holds.  

Again since $Z_i$ is a crepant resolution of $\bbC^k/\Gamma$, we know by \cite{Ito-Reid, Batyrev, Denef-Loeser} that it has trivial cohomology in odd degree.  By Poincar\'e duality, we deduce from \cite[Theorem~1A]{HHM2004} that $Z_i\setminus \pa Z_i$ with its induced $\ALE$-metric has only possibly non-trivial $L^2$-cohomology in middle degree when $\dim Z_i=4$.  This shows that the last part of Assumption~\ref{do.5e} holds.  For the first part of this assumption, notice that $S_i$ is a manifold with fibered corners resolving an orbifold quotient of an odd dimensional sphere.  In particular, by Assumption~\ref{do.1}, for $H_i$ maximal, \eqref{qfb.12a} automatically holds.  For $H_i$ submaximal, it holds if $\dim S_i\ge3$, while if $\dim S_i=1$ then it holds trivially since we assume in this case that $\ker \mathfrak{d}_{S_i}$ is trivial.  Hence, we see that Assumption~\ref{do.1} and Assumption~\ref{do.5e} are fulfilled and Theorem~\ref{do.6} holds.  
\end{proof}

This result provides an alternative description to the one of Carron \cite[Théorème~E]{Carron2011b}, where reduced $L^2$-cohomology of a $\QALE$ metric $g_{\QALE}$ is computed in terms of the cohomologies of different complexes of differential forms.  In fact, when $M\setminus \pa M$ is a crepant resolution of $\bbC^4/\Gamma$ with $\Gamma$ a finite of $\Sp(2)$, one recovers exactly \cite[Théorème~7.14]{Carron2011b} from Corollary~\ref{do.8} as the next corollary shows.

\begin{corollary}[Carron]
If $g_{\QALE}$ is a $\QALE$-metric on a crepant resolution of $\bbC^n/\Gamma$ where  $n=4$ and $\Gamma$ is a subgroup of $\Sp(2)$, then
$$
L^2\cH^q(M\setminus \pa M, g_{\QALE})\cong \left\{ \begin{array}{ll} H^q_{c}(M\setminus \pa M), & q<n, \\
      \Im (H^q_c(M\setminus \pa M)\to H^q(M)), & q=n, \\
      H^q(M), & q>n.   \end{array} \right.
$$  
\label{do.9}\end{corollary} 
\begin{proof}
In this case the fibers $Z_i$ of $\phi_i: H_i\to S_i$ are 4-dimensional whenever $H_i$ is submaximal, so $\dim S_i=3$ and Corollary~\ref{do.8} applies.  As discussed in \cite{Carron2011b}, we know from \cite{Ito-Reid, Batyrev, Denef-Loeser}  that the cohomology of a crepant resolution of $\bbC^k/\Gamma$ for $\Gamma$ a finite subgroup of $\SU(k)$ has trivial cohomology in odd degrees.  This means that $H^q(Z_i)$ is non-trivial only for $q\in\{0,2\}$.  In particular, from Lemma~\ref{do.2}, Proposition~~\ref{do.4} and Theorem~\ref{wl2.26}, we see that $\WH^*_{\QCyl}(M,\phi, -d)$ has the same local behavior as $H^*(M\setminus\pa M)$,  for $d=(\delta,\ldots,\delta)\in\bbR^{\ell}$ with $\delta>0$ such that $\delta\notin \bbN$.  Using commutative diagrams of Mayer-Vietoris long exact sequences, one can therefore show as in the proof of Corollary~\ref{sa.10} that 
\begin{equation}
     \WH^q_{\QCyl}(M,\phi,-d)\cong H^q(M\setminus \pa M)
\label{do.9b}\end{equation}
for $d$ as described above.  Applying a similar argument for $-d$ or directly applying Poincaré duality to the isomorphism \eqref{do.9b}, one can check that there is also an isomorphism
\begin{equation}
     \WH^q_{\QCyl}(M,\phi,d)\cong H^q_c(M\setminus \pa M).
\label{do.9bb}\end{equation}
Hence, for $a=(\epsilon,\ldots,\epsilon)\in \bbR^{\ell}$ with $0<\epsilon<\frac12$ and $q\ne n$,
we see that
\begin{equation}
 \WH^q_{\QAC}(M,\phi,\pm a)=\WH^q_{\QCyl}(M,\phi,\pm a+ \underline{\left(n-q\right)})\cong\left\{
 \begin{array}{ll} H^q_c(M\setminus \pa M), & q<n, \\
        H^q(M), & q>n,  \end{array} \right.
\label{do.9c}\end{equation}
while for $q=n$, we have instead
\begin{equation}
\begin{gathered}
  \WH^q_{\QAC}(M,\phi, a)=\WH^q_{\QCyl}(M,\phi, a)= H^q_c(M\setminus \pa M), \\
  \WH^q_{\QAC}(M,\phi, -a)=\WH^q_{\QCyl}(M,\phi, -a)= H^q(M\setminus \pa M). 
  \end{gathered}
  \label{do.9d}\end{equation}
Plugging \eqref{do.9c} and \eqref{do.9d} in \eqref{do.8b} then gives the result. 
\end{proof}
A similar result was obtained by Carron \cite[Théorème~7.12]{Carron2011b} when $g_{\QALE}$ is a $\QALE$-metric on a crepant resolution of $\bbC^3/\Gamma$ for $\Gamma$ a finite subgroup of $\SU(3)$.  In this case however, for $H_i$ submaximal, $\dim S_i=1$, so unless we are in the special case where $\ker \mathfrak{d}_{S_i}=\{0\}$, we cannot deduce this result from Corollary~\ref{do.8}.  Notice however that the proof of Corollary~\ref{do.9} shows that the statement of \cite[Théorème~7.12]{Carron2011b} is consistent with the formulation provided in \eqref{do.8b}.

\section{$L^2$-cohomology of the moduli space of monopoles} \label{mms.0}

We can also apply Corollary~\ref{wl2.35} to the moduli space of $\SU(2)$-monopoles of magnetic charge $3$.  More precisely, let $\cM_k$ denote the moduli space of $\SU(2)$-monopoles of magnetic charge $k$ on $\bbR^3$,  let $\cN_k=\cM_k/\bbR^3$ be the corresponding space of centered monopoles and let $\cM^0_k= \cN_k/\bbS^1$ be the reduced moduli space.  Finally, let $\tcM^0_k$ be the universal cover of $\cM^0_k$.  When $k=2$, we know \cite[\S~7.1.2]{HHM2004} that
$$
     \cM^0_2\cong S^4\setminus \bbR\bbP^2 \quad \mbox{and} \quad \tcM^0_2\cong \bbC\bbP^2\setminus \bbR\bbP^2.
$$  
In this case, the natural $L^2$-metric on $\cM^0_2$ is a fibered boundary metric.  Let $g_{\fc}$ be a conformally related fibered cusp metric.  Let $x$ be a boundary defining function for the associated boundary compactification. Since the associated stratified spaces, which are respectively $S^4$ and $\bbC\bbP^2$, are smooth and that intersection cohomology does not depend on the choice of stratification, we know by  \cite[Proposition~2]{HHM2004} that for $0<\epsilon<1$, the weighted $L^2$-cohomologies $g_{\fc}$ and its lift $\widetilde{g}_{\fc}$ to $\tcM^0_2$ are given by
\begin{equation} 
         \WH^*(\cM^0_2,g_{\fc},x^{\pm\epsilon})\cong H^*(S^4),  \quad \WH^*(\tcM^0_2,\widetilde{g}_{\fc},x^{\pm\epsilon})\cong H^*(\bbC\bbP^2).
\label{mms.1}\end{equation}

Now, by \cite[p.34]{Atiyah-Hitchin}, we know that the circle bundle $\cN_k\to \cM^0_k$ is flat, so that its lift to $\tM^0_k$ is trivial.  In other words, $\cN_k$ admits a $k$-fold cover 
\begin{equation}
      \widetilde{\cN}_k\cong \bbS^1\times \tcM^0_k.
\label{mms.1b}\end{equation}
\begin{lemma}
For $0<\epsilon<\frac12$ and for the fibered cusp metric conformally related to the natural $L^2$-metric on $\cN_2$ and $\tcN_2$, we have that
$$
      \WH^q_{\QFC}(\tcN_2,\phi,\pm\epsilon)\cong \bbR \quad \forall q\in\{0,1,2,3,4,5\},  \quad \WH^q_{\QFC}(\cN_2,\phi,\pm\epsilon)\cong\left\{ \begin{array}{ll} \bbR, & q\in\{0,1,4,5\}, \\
      \{0\}, & q\in\{2,3\}. \end{array} \right.
$$
\label{mms.2}\end{lemma}
\begin{proof}
On $\widetilde{\cN}_2\cong \bbS^1\times \tcM^0_2$,  the $L^2$-metric is a fiber boundary metric with the factor $\bbS^1$ part of the fiber with respect to the fiber bundle on the boundary.  Thus, for the purpose of computing the weighted $L^2$-cohomology of a conformally related fibered cusp metric, we can consider a fibered cusp metric of the form
\begin{equation}
          g_{\tcN_2}=  g_{\tcM^0_2} + x^2g_{\bbS^1}
\label{mms.3}\end{equation}
where $g_{\tcM^0_2}$ is the corresponding fibered cusp metric on $\tcM^0_2$ and $x$ is the boundary defining function of the $\QFB$-compactification on $\tcM_2^0$.  In particular, since \eqref{mms.3} is a warped product and the exterior differential $d$ on $\bbS^1$ is closed, we can apply \cite[Theorem~2.29]{Zucker} to compute weighted $L^2$-cohomology.      Combined with \eqref{mms.1}, this gives
\begin{equation}
\begin{aligned}  
\WH^q_{\QFC}(\tcN_2,\phi,\pm\epsilon) &\cong  \WH^q_{\QFC}(\tcM^0_2,\phi,\pm\epsilon-\frac12)\oplus \WH^{q-1}_{\QFC}(\tcM^0_2,\phi,\pm\epsilon+\frac12) \\
    &\cong H^{q}(\bbC\bbP^2)\oplus H^{q-1}(\bbC\bbP^2)
\end{aligned}
\label{mms.4}\end{equation} 
as claimed.  For $\cN_2$, it suffices to take the $\bbZ_2$-invariant part of \eqref{mms.4} under the action of $\bbZ_2$ and to notice that $H^2(\bbC\bbP^2)$ has no non-trivial $\bbZ_2$-invariant, this latter fact being a consequence of the Sen conjecture  on $\tM^0_2$ proved by Hitchin \cite{Hitchin}.  
\end{proof}

\begin{proposition}
Let $(\bcM^0_3,\phi)$ be the  $\QFB$-compactification constructed in \cite{FKS} of $\tcM^0_3$ and let $\hcM^0_3$ be the corresponding stratified space.  Fix $0<\epsilon<\frac12$.  If $a$ is a multiweight such that $a_i=\epsilon$ for all $i$, then 
\begin{equation}
   \WH_{\QFC}^*(\bcM^0_3,\phi, a)\cong \IH^*_{\overline{\mathfrak{m}}}(\hcM^0_3) \quad \mbox{and} \quad \WH_{\QFC}^*(\bcM^0_3,\phi, -a)\cong \IH^*_{\underline{\mathfrak{m}}}(\hcM^0_3). 
\label{mms.5e}\end{equation}
Moreover, for $q=\frac{m}2=\frac{\dim \bcM^0_3}2=4$, the space $\WH^4_{\QFB}(\bcM^0_3,\phi,a)$ is Poincaré dual to $\WH^4_{\QFB}(\bcM^0_3,\phi,-a)$ and there are natural maps
\begin{gather}
\label{mms.5b}  \WH^4_{\QFB}(\bcM^0_3,\phi, a)\to L^2\cH^4(\tcM^0_3 , g_{\QFB}), \\
\label{mms.5c} L^2\cH^4(\tcM^0_3, g_{\QFB})\to \WH^4_{\QFB}(\bcM^0_3,\phi,-a),
\end{gather}
with composition
\begin{equation}
\WH^4_{\QFB}(\bcM^0_3,\phi, a)\to L^2\cH^4(\tcM^0_3, g_{\QFB})\to \WH^4_{\QFB}(\bcM^0_3,\phi,-a)
\label{mms.5d}\end{equation}
corresponding to the obvious map, where $g_{\QFB}$ is a choice of $\QFB$-metric compatible with the manifold with fibered corners $(\bcM^0_3,\phi)$.    

\label{mms.5}\end{proposition}
\begin{proof}
As described in \cite{FKS}, the manifold with fibered corners $\bcM^0_3$ has  two boundary hypersurfaces.  Denoting them $H_1$ and $H_2$ with $H_1<H_2$, we know that the fiber $Z_2$ of $H_2$ is a $\bbZ_3$-cover of a $2$-dimensional torus.  In particular, $Z_2$ is a closed manifold and the assumption of Corollary~\ref{wl2.35} is automatically satisfied for $i=2$.  For $i=1$, the typical fiber $Z_1$ of $\phi_1: H_1\to S_1$ is a $\bbZ_3$-cover of $\cN_2$.  

Now, by the discussion in \cite[p.20]{Atiyah-Hitchin}, $\pi_1(\cN_2)\cong \bbZ$ and a generator of this group is sent to a generator of $\pi_1(\cM^0_3)\cong \bbZ_3$ under the inclusion 
$\cN_2\to \cM^0_3$.  Hence, on the universal cover $\tcM^0_3$, $\cN_2$ lifts to a (connected) $\bbZ_3$-cover of $\cN_2$.  We claim that this cover is in fact $\cN_2$ itself.  Assuming this, then Lemma~\ref{mms.2} ensures that the assumption of Corollary~\ref{wl2.35} is also satisfied for $i=1$, so that \eqref{mms.5e} follows by applying Corollary~\ref{wl2.35}.

To prove the remaining claim, consider the long exact sequence in homotopy groups induced by the circle bundle $\cN_2\to \cM^0_2$.  In particular, we see from \cite[p.20]{Atiyah-Hitchin} that it induces a short exact sequence
$$
     0\to \pi_1(\bbS^1)\to \pi_1(\cN_2)\to \pi_1(\cM^0_2)\to 0
$$  
which is just the standard exact sequence
$$
        0\to \bbZ\overset{\times 2}{\to} \bbZ\to \bbZ_2\to 0.
$$
Now, the $\bbZ_3$-cover $\cW_2$ of $\cN_2$ corresponds to the normal subgroup $3\bbZ$ of $\pi_1(\cN_2)\cong \bbZ$.  Hence, by standard homotopy lifting properties, under the canonical projection $\nu: \cW_2\to \cN_2$, the pre-image of a fiber $F$ of $\cN_2\to \cM^0_2$ is a circle which is a $\bbZ_3$-cover of $F$.  In other words, if $L$ is the complex line bundle associated to the circle bundle $\cW_2\to \cM^0_2$, then $\cN_2\to \cM^0_2$ is the circle bundle associated to the line bundle $L\otimes L\otimes L$.  However, according to \cite[Theorem~1.1]{Segal-Selby}, $\cM^0_2$ has the rational homology of a point.  Hence, since $\pi_1(\cM^0_2)\cong \bbZ_2$, we conclude by the universal coefficient theorem that $H^2(\cM^0_2)\cong \bbZ_2$,  so that $L\otimes L\otimes L\cong L$. This means that $\cW_2\cong \cN_2$ as claimed. 

When $q=\frac{m}2=4$, Poincar\'e duality follows from Corollary~\ref{sa.8} and the fact that in middle degree
$$
      \WH^4_{\QFB}(\bcM^0_3,\phi, \pm a)= \WH^4_{\QFC}(\bcM^0_3,\phi, \pm a).
 $$
For the definition of the maps \eqref{mms.5b} and \eqref{mms.5c} and the proof that the composition \eqref{mms.5d} is the obvious map, we can proceed exactly as in the proof of Proposition~\ref{do.5}.
\end{proof}

Combining this with the decay of harmonic forms of Theorem~\ref{qfb.12} yields the following.  
\begin{theorem}
Let $g_{\QFB}$ be a choice of $\QFB$-metric compatible with $(\bcM^0_3,\phi)$.  Then there is a natural identification
\begin{equation}
\begin{aligned}
L^2\cH^4(\tcM^0_3,g_{\QFB}) &\cong \Im \left( \WH^4_{\QFB}(\bcM^0_3,\phi,a)\to \WH^4_{\QFB}(\bcM^0_3,\phi,-a), \right) \\
& \cong \Im \left( \IH^4_{\overline{\mathfrak{m}}}(\hcM^0_3)\to \IH^4_{\underline{\mathfrak{m}}}(\hcM^0_3), \right) \end{aligned}
\label{mms.6b}\end{equation}
where the multiweight $a:=(\epsilon,\epsilon)$ for $0<\epsilon<\frac12$ is as in Proposition~\ref{mms.5}.
\label{mms.6}\end{theorem}
\begin{proof}
By Proposition~\ref{mms.5}, there is an inclusion 
\begin{equation}
\Im(\WH^4_{\QFB}(\bcM^0_3,\phi,a)\to \WH^4_{\QFB}(\bcM^0_3,\phi,-a))\hookrightarrow L^2\cH^4(\tcM^0_3,g_{\QFB}),
\label{mms.7}\end{equation}
so it suffices to show that this inclusion is surjective.  Since the statement only depends on the quasi-isometry class of $g_{\QFB}$, to show this, we can choose $g_{\QFB}$ as we want within this class.  Now, if all the hypotheses of Theorem~\ref{qfb.12} are fulfilled, then we know by this theorem that $g_{\QFB}$ can be chosen so that 
\begin{equation}
    L^2\cH^4(\tcM^0_3,g_{\QFB})\subset (x_1x_2)^{\mu}L^2\Omega^4(\tcM^0_3,g_{\QFB})
\label{mms.8}\end{equation}
for some $\mu>0$,  so that all $L^2$-harmonic forms come from the inclusion \eqref{mms.7} as claimed.  

Thus, to complete the proof, it suffices to show that Theorem~\ref{qfb.12} can be applied.  
First, by the discussion in \cite{FKS}, the base $S_1$ is $\bbR\bbP^2$, while $S_2$ is the manifold with fibered corners (in fact more simply the manifold with fibered boundary) resolving the quotient 
\begin{equation}
  \widehat{S}_2= \bbS^5/\mathbf{S}_3
\label{mms.9}\end{equation}   
of $\bbS^5\subset \bbR^6$ by the action of the symmetric group $\mathbf{S}_3$ on $\bbR^6$ generated by
\begin{equation}
\left(\begin{array}{cc} 0 &\mathbb{I}_3  \\ \mathbb{I}_3 & 0 \end{array} \right) \quad \mbox{and} \quad
 \left(\begin{array}{cc} 0 &-\mathbb{I}_3  \\ \mathbb{I}_3 & -\mathbb{I}_3 \end{array} \right).
 \label{mms.10}\end{equation}
Under this action, the subsets of points of $\bbS^5$ where the action is not free correspond to three disjoint 2-spheres $\bbS^2$ in $\bbS^5$.  Under the quotient map, these three $2$-spheres are mapped onto the singular stratum of $\widehat{S}_2$, which is just $S_1=\bbR\bbP^2$.  Moreover, the link of the stratum is also diffeomorphic to $\bbR\bbP^2$.  Since $\bbR\bbP^2$ has no cohomology in middle degree, condition \eqref{qfb.12aa} is satisfied for $H_1<H_2$.  Moreover, since $\widehat{S}_2$ is a quotient of $\bbS^5$, taking the wedge metric  $g_{w}$ induced by the standard metric on $\bbS^5$, we see that $\widehat{S}_2$ has no non-trivial $L^2$-harmonic forms in degrees
$\frac{\bd_2\pm1}2= \frac{5\pm 1}2$.  Obviously, for dimensional reasons, there are no $L^2$-harmonic forms in degrees $\frac{\bd_2}2=\frac52$ and $\frac{\bd_2\pm 2}2$, so that condition \eqref{qfb.12a} holds for $H_2\le H_2$.    On the other hand, let $\widetilde{g}_2$ be the natural hyperK\"ahler metric on $\tcM^0_2$, which is known to be a fibered boundary metric.  From \cite{Hitchin}, we know that 
\begin{equation}
  L^2\cH^q(\tcM^0_2,\widetilde{g}_2)\cong \left\{\begin{array}{ll} \bbR, & q=2, \\
  \{0\}, & \mbox{otherwise}.\end{array}  \right.
\label{mms.11}\end{equation}
Correspondingly, the natural fibered boundary metric on $\widetilde{\cN}_2= \bbS^1\times \tcM^0_2$ is the Cartesian product metric $g_{\bbS^1}+ \widetilde{g}_{2}$ with $g_{\bbS^1}$ the standard metric on $\bbS^1$.  Using separation of variables, we obtain from \eqref{mms.11} that 
\begin{equation}
  L^2\cH^q(\tcN_2,g_{\bbS_1}+\widetilde{g}_2)\cong \left\{\begin{array}{ll} \bbR, & q=2,3, \\
  \{0\}, & \mbox{otherwise},\end{array}  \right.
\label{mms.12}\end{equation}
a result that follows alternatively from \cite[Corollary~1]{HHM2004}.
Since the corresponding $L^2$-harmonic forms are not invariant under the natural $\bbZ_2$-action by the Sen conjecture for $\tcM^0_2$, this means that  the quotient $\cN_2= \tcN_2/\bbZ_2$ has no non-trivial $L^2$-harmonic forms with respect to the induced metric $g_{\cN_2}$,
 \begin{equation}
   L^2\cH^q(\cN_2, g_{\cN_2})=\{0\}.
\label{mms.13}\end{equation}  
In particular, this means that condition \eqref{qfb.12b} is autmatically satisfied and that the operator $\mathfrak{d}_{S_1}$ is trivial, which implies that condition \eqref{qfb.12a} for $H_1\le H_1$ is automatically satisfied.  
By Theorem~\ref{qfb.12}, this means that we can find a $\QFB$-metric $g_{\QFB}$ such that \eqref{mms.8} holds, completing the proof of the theorem.
\end{proof}

In middle degree, the intersection cohomology with lower or upper middle perversity can in fact be described in terms of the usual cohomology as the next proposition shows.
\begin{proposition}
The natural maps 
$$
    H_c^4(\tcM^0_3)\to \IH^4_{\overline{\mathfrak{m}}}(\hcM^0_3) \quad \mbox{and} \quad \IH^4_{\underline{\mathfrak{m}}}(\hcM^0_3)\to H^4(\tM^0_3)    
$$
are isomorphisms.
\label{mms.14}\end{proposition}
\begin{proof}
Since these two maps are Poincar\'e dual to one another, it suffices to prove that the first one is an isomorphism.  Now, by \cite[Theorem~1.3]{Segal-Selby}, the composition 
$$
H_c^4(\tcM^0_3)\to \IH^4_{\overline{\mathfrak{m}}}(\hcM^0_3)\to H^4(\tcM^0_3)
$$ 
is an isomorphism, implying that the map
\begin{equation}
H_c^4(\tcM^0_3)\to \IH^4_{\overline{\mathfrak{m}}}(\hcM^0_3)
\label{mms.15}\end{equation}
is injective.  To show it is surjective will require more work and local computations.  Let $\cU$ be a collar neighborhood of $H_1$ in $\bcM^0_3$.  Denote by $\hU$ the corresponding neighborhood of the stratum $S_1$ in $\hcM^0_3$.  Let $\hV$ be an open neighborhood of $\widehat{S}_2\setminus (\hU\cap \widehat{S}_2)$ which retracts onto it.  With these choices, we can assume as well that $\tcM^0_3$ retracts onto $\hcM^{0}_3\setminus (\hU\cup \hV)$, so that
\begin{equation}
    \IH^q_{\overline{m}}(\hcM^0_3, \hU\cup \hV)\cong H^q_{c}(\tcM^0_3).
\label{mms.16}\end{equation}
Hence, from the relative long exact sequence in cohomology
\begin{equation}
\xymatrix{
\cdots \ar[r] & \IH^q_{\overline{\mf{m}}}(\hcM^0_3, \hU\cup\hV) \ar[r] &  \IH^q_{\overline{\mf{m}}}(\hcM^0_3) \ar[r] &
   \IH^q_{\overline{\mf{m}}}(\hU\cup\hV) \ar[r] & \cdots, 
}
\label{mms.17}\end{equation}
we see that \eqref{mms.15} will be surjective provided we can show that $\IH^4_{\overline{\mf{m}}}(\hU\cup\hV)=\{0\}$.  To see this, we will first compute $\IH^4_{\overline{\mf{m}}}(\hU)$.  Recall that by the proof of Proposition~\ref{mms.5}, the fibers of the fiber bundle $\phi_1: H_1\to S_1$ are each diffeomorphic to $\cN_2$.  Since $\hU$ comes from a collar neighborhood of $H_1$ in $\bcM^0_3$, this means that there is a corresponding fiber bundle
\begin{equation}
   \hat{\phi}_1: \hU\to S_1
\label{mms.18}\end{equation}  
whose fibers are cones $C(\hcN_2)$ over $\hcN_2$, the stratified space associated to the manifold with fibered boundary $\cN_2$.  According to \eqref{wl2.36} and \cite[Proposition~1]{HHM2004}, we have that
\begin{equation}
   \IH^q_{\overline{\mf{m}}}(C(\hcN_2))\cong \left\{ \begin{array}{ll} \IH^q_{\overline{\mf{m}}}(\hcN_2), & q\le2, \\
      \{0\}, & \mbox{otherwise.}  \end{array} \right.
\label{mms.19}\end{equation}
By Lemma~\ref{mms.2} and \cite[Proposition~2]{HHM2004}, we thus have that
\begin{equation}
 \IH^q_{\overline{\mf{m}}}(C(\hcN_2))\cong \left\{ \begin{array}{ll} \bbR, & q\le 1, \\
      \{0\}, & \mbox{otherwise.}  \end{array} \right.
\label{mms.20}\end{equation}

Now $S_1\cong \bbR\bbP^2$ with universal cover $\widetilde{S}_1\cong \bbS^2$.  Let 
\begin{equation}
  \widetilde{\hat{\phi}}_1: \widetilde{\hU}\to \widetilde{S}_1
\label{mms.21}\end{equation}
denote the pull-back of the fiber bundle \eqref{mms.18} to $\widetilde{S}_1$.  The space $\widetilde{\hU}$ is then a $\bbZ_2$-cover of $\hU$ with $\bbZ_2$-action covering the $\bbZ_2$-action on $\widetilde{S}_1$.  This implies in particular that $\IH_{\overline{\mf{m}}}^q(\hU)$ corresponds to the $\bbZ_2$-invariant part of $\IH_{\overline{\mf{m}}}^q(\widetilde{\hU})$.   By \eqref{mms.19}, we see as in \cite[Theorem~14.18]{Bott-Tu} that the second page of the Leray spectral sequence of \eqref{mms.21} for intersection cohomology with upper middle perversity is given by 
\begin{equation}
   E^{p,q}_2= H^p(\widetilde{S}_1)\otimes \IH_{\overline{\mf{m}}}^q(C(\hcN_2))= \left\{ \begin{array}{ll}\bbR, & p\in\{0,2\}, \; q\in\{0,1\}, \\
   \{0\}, & \mbox{otherwise.}  \end{array} \right.
\label{mms.22}\end{equation}
In particular, as in \cite[Example~14.22]{Bott-Tu}, for dimensional reasons, the spectral sequence degenerates at page $E_3$, allowing us to conclude that 
$$
\IH^q_{\overline{\mf{m}}}(\widetilde{\hU})=\{0\} \quad \mbox{for} \; q\ge 4.
$$
Taking the $\bbZ_2$-invariant part thus implies that 
\begin{equation}
   \IH^q_{\overline{\mf{m}}}(\hU)=\{0\} \quad \mbox{for} \; q\ge 4.
\label{mms.23}\end{equation}

On $\hV$, the local behavior of intersection cohomology for upper middle perversity also admits a simple description.  Indeed, the fibers of $\phi_2: H_2\to S_2$ are $2$-dimensional tori.  Hence, from \eqref{wl2.36}, given $p\in \widehat{S}_2\cap \hV$ and a small neighborhood $\hV_p$ of $p$ in $\hV$ retracting onto $p$, we have that 
\begin{equation}
 \IH^q_{\overline{\mf{m}}}(\hV_p)\cong \left\{ \begin{array}{ll} \bbR\cong H^0(\bbT^2), & q=0, \\
    \{0\}, & q>0.  \end{array} \right.  
\label{mms.24}\end{equation}
Since $\hV$ retracts onto  $\hV\cap \widehat{S}_2$, this means that 
\begin{equation}
   \IH^q_{\overline{\mf{m}}}(\hV)= H^q(\hV)= H^q(\hV\cap \widehat{S}_2).
\label{mms.25}\end{equation}
Now, from \eqref{mms.9} and \eqref{mms.10}, we know that the universal cover of $\hV\cap \widehat{S}_2$ is homeomorphic to $\bbS^5\setminus \Sigma$ with $\Sigma$ a disjoint union of three $2$-spheres $\bbS^2$ inside $\bbS^5$.  In particular, from the relative long exact sequence in cohomology  
\begin{equation}
\xymatrix{
\cdots \ar[r] & H^q(\bbS^5,\Sigma) \ar[r] & H^q(\bbS^5) \ar[r] & H^q(\Sigma) \ar[r] & \cdots
}
\label{mms.26}\end{equation} 
associated to the pair $(\bbS^5,\Sigma)$, we deduce that 
\begin{equation}
\IH^4_{\overline{\mf{m}}}(\hU\cup \hV, \hU)\cong\IH^4_{\overline{\mf{m}}}(\hV, \hV\cap \hU)\cong H^4_c(\hV\cap \widehat{S}_2)=\{0\}.
\label{mms.27}\end{equation}
Hence, from \eqref{mms.23} and \eqref{mms.27} and the relative long exact sequence in cohomology
\begin{equation}
\xymatrix{
\cdots \ar[r] & \IH^q_{\overline{\mf{m}}}(\hU\cup \hV, \hU)\ar[r] & \IH^q_{\overline{\mf{m}}}(\hU\cup \hV)\ar[r] & \IH^q_{\overline{\mf{m}}}(\hU)\ar[r] & \cdots
}
\label{mms.28}\end{equation}
associated to the pair $(\hU\cup\hV,\hU)$, we deduce that
\begin{equation}
 \IH^4_{\overline{\mf{m}}}(\hU\cup\hV)= \{0\}
\label{mms.29}\end{equation}
as claimed.
\end{proof}

Let $\widetilde{g}_3$ be the natural hyperK\"ahler metric on $\tcM^0_3$.  In \cite{FKS}, it was announced that $\widetilde{g}_3$ is a $\QFB$-metric with respect to the manifold with fibered corners $(\bcM^0_3,\phi)$.  Assuming this result, we can extract from Theorem~\ref{mms.6} and Proposition~\ref{mms.14} a proof of the Sen conjecture \cite{Sen, Segal-Selby}.

\begin{theorem}
The Sen conjecture holds on $\tcM^0_3$, namely
\begin{equation}
     L^2\cH^q(\tcM^0_3,\widetilde{g}_3)\cong \Im (H^q_c(\tcM^0_3)\to H^q(\tcM^0_3))
\label{mms.30b}\end{equation}
for all $q\in\bbN_0$.  
\label{mms.30}\end{theorem}
\begin{proof}
By the work of Hitchin \cite{Hitchin}, the result holds for $q\ne4$.  For $q=4$, the isomorphism \eqref{mms.30b} follows from Theorem~\ref{mms.6} and Proposition~\ref{mms.14}, thanks to the result announced in \cite{FKS} that $\widetilde{g}_3$ is a $\QFB$-metric.   
\end{proof}

Unfortunately, the same approach does not seem to work to prove the Sen conjecture on $\tcM^0_4$.  As the next lemma indicates, the main problem is that on $\tcM^0_4$, one of the assumptions of Corollary~\ref{wl2.35} does not hold, compromising the use of this corollary to obtain an analogue of Proposition~\ref{mms.5} for $\tcM^0_4$.
\begin{lemma}
Fix $0<\epsilon <\frac12$.  If $a$ is a multiweight such that $a_i=\epsilon$ for all $i$, then the natural map
\begin{equation}
\WH^5_{\QFC}(\cN_3,\phi,\epsilon)\longrightarrow H^5(\cN_3)
\label{ce.1a}\end{equation}
is not injective.
\label{ce.1}\end{lemma}
\begin{proof}
By \eqref{mms.1b},
$$
              H^5(\tcN_3)\cong H^4(\tcM^0_3),
$$
while by  \cite{Segal-Selby}, $H^4(\tcM^0_3)$ has no $\bbZ_3$-invariant part, so
$$
    H^5(\cN_3)=\{0\}.
$$
Hence to show that \eqref{ce.1a} is not injective, it suffices to show that $\WH^q_{\QFC}(\cN_3,\phi,\epsilon)$ is non-trivial.  Applying the same idea as in the proof of Lemma~\ref{mms.2}, we have that
\begin{equation}
\begin{aligned}
\WH^5_{\QFC}(\tcN_3,\phi,\epsilon) &\cong \WH^5_{\QFC}(\bcM^0_3,\phi,\epsilon-\frac12) \oplus \WH^4_{\QFC}(\bcM^0_3,\phi,\epsilon+\frac12) \\
   &\cong  \WH^5_{\QFC}(\bcM^0_3,\phi,-\epsilon) \oplus \WH^4_{\QFC}(\bcM^0_3,\phi,\epsilon+\frac12) \\
   &\cong \IH^5_{\underline{\mathfrak{m}}}(\hcM^0_3)\oplus \WH^4_{\QFC}(\bcM^0_3,\phi,\epsilon+\frac12).
\end{aligned}
\label{ce.2}\end{equation}
Now, by \eqref{mms.26} and \eqref{mms.27}, notice that 
$$
     \IH_{\overline{\mathfrak{m}}}^3(\hU\cup \hV,\hU)\cong \bbC^3,
$$
so by \eqref{mms.22} and \eqref{mms.28}, 
$$
    \dim_{\bbC}\IH^3_{\overline{\mathfrak{m}}}(\hU\cup\hV)\ge 2.
$$
Hence, by \eqref{mms.17}, the fact that $\IH^3_{\overline{\mathfrak{m}}}(\hcM^0_3,\hU\cup\hV)=H^3_c(\tcM^0_3)=\{0\}$ by \cite{Segal-Selby} and the fact that the map 
$$
     \IH^4_{\overline{\mathfrak{m}}}(\hcM^0_3,\hU\cup\hV)\longrightarrow \IH^4_{\overline{\mathfrak{m}}}(\hcM^0_3)
$$
is injective by Proposition~\ref{mms.14}, we see that $\dim_{\bbC}\IH^3(\hcM^0_3)\ge 2$.  By duality, this means that $\dim_{\bbC}\IH^5_{\underline{\mathfrak{m}}}(\hcM^0_3)\ge 2$, which by \eqref{ce.2} implies that
$$
\dim_{\bbC}\WH^5_{\QFC}(\tcN_3,\phi,\epsilon) \ge 2.
$$
Since these non-trivial elements of $\WH^5_{\QFC}(\tcN_3,\phi,\epsilon)$  `come from' $ \IH_{\overline{\mathfrak{m}}}^3(\hU\cup \hV,\hU)$, they are automatically $\bbZ_3$-invariant, so we can conclude that
$$
\dim_{\bbC}\WH^5_{\QFC}(\cN_3,\phi,\epsilon) \ge 2.
$$
\end{proof}

\bibliography{QFBop}
\bibliographystyle{amsplain}

\end{document}